\documentclass[11pt]{amsart}
\usepackage{amssymb}

\usepackage{amsmath}
\usepackage{amsthm}
\usepackage{a4wide}
\usepackage{pdfsync}
\usepackage{hyperref}
\usepackage{color}

 \usepackage{mathtools}

\mathtoolsset{showonlyrefs} 

\newtheorem{lemma}{Lemma}[section]
\newtheorem{theorem}{Theorem}[section]
\newtheorem{remark}{Remark}[section]

\newcommand{\sgn}{\,{\rm sgn}}

\def\rr{\mathbb{R}}

\def\eps{\varepsilon}

\def\bu{{\bf{u}}}
\def\bU{{\bf{U}}}
\def\bw{{\bf{w}}}
\def\bv{{\bf{v}}}
\def\ue{u^e}
\def\ve{v^e}

\newcommand{\LL}{L}

\title[Local and nonlocal evolution equations on graphs]{Asymptotic behaviour for local and nonlocal evolution equations on
metric graphs with some edges of infinite length}
\author[L. I. Ignat \and J. D. Rossi \and A. San Antolin ]{Liviu I. Ignat \and Julio D. Rossi \and Angel San Antolin
}

\address{L. I. Ignat
\hfill\break\indent Institute of Mathematics ``Simion Stoilow'' of the Romanian Academy,
Centre Francophone en Math\'{e}matique
\\21 Calea Grivitei Street \\010702 Bucharest, ROMANIA.
 \hfill\break\indent  ICUB,   The Research Institute of the University of Bucharest, University of Bucharest\\
36-46 Bd. M. Kogalniceanu,   050107, Bucharest, ROMANIA.
}
 \email{{\tt
liviu.ignat@gmail.com}  \hfill\break\indent {\it Web page: }{\tt
http://www.imar.ro/\~\,lignat}}

\address{J. D. Rossi
\hfill\break\indent Dpto. de Matem{\'a}ticas, FCEyN,
Universidad de Buenos Aires, \hfill\break\indent 1428, Buenos Aires,
ARGENTINA. } \email{{\tt jrossi@dm.uba.ar} \hfill\break\indent {\it
Web page: }{\tt http://mate.dm.uba.ar/$\sim$jrossi/}}

\address{A. San Antolin
\hfill\break\indent Departamento de An\'{a}lisis Matem\'{a}tico,
 Universidad de Alicante,
\hfill\break\indent Ap. correos 99, 03080,
Alicante, SPAIN. } \email{{\tt angel.sanantolin@ua.es}}

\begin{document}


\begin{abstract} We study local (the heat equation) and nonlocal (convolution type problems
with an integrable kernel) evolution problems on a metric connected finite
graph in which some of
the edges have infinity length. We show that the asymptotic behaviour of the solutions to both local and nonlocal problems is given by the solution
of the heat equation, but on a star shaped graph in which there is only one node
and as many infinite edges as in the original graph.
In this way we obtain that the
compact component that consists in all the vertices and all the edges of finite length can be reduced
to a single point when looking at the asymptotic behaviour of the solutions.
For this star shaped limit problem the asymptotic behaviour
of the solutions is just given by the solution to the heat equation in a half line with a Neumann boundary condition at $x=0$ and initial
datum $(2 M/N ) \delta_{x=0}$ where $M$ is the total mass of the initial condition for our original problem and $N$ is the number of
edges of
infinite length.
In addition, we show that solutions to the nonlocal problem converge, when we rescale the kernel, to solutions to the heat equation
(the local problem), that is, we find a relaxation limit.
\end{abstract}

\maketitle

\section{Introduction} \label{sect-intro}

The aim of this paper is to study solutions to diffusion equations both local and non-local in a
metric graph. A metric graph is by definition a combinatorial graph where
the edges, denoted by $\{e_j\}$ are considered as intervals of the real line $\{I_j\}$ with a distance on each one of them.
These edges/intervals are glued together according to the combinatorial
structure. We assume here that at least one of the edges is not bounded (it has infinite length).

Metric graphs have received lot of attention in recent years
both from the point of view of pure mathematicians and also from potential applications.
The name quantum graph is used for a graph considered as an one-dimensional
singular variety and equipped with a differential operator (local or in some cases nonlocal).
There are several reasons for studying
quantum graphs. They naturally arise as simplified (due to reduced dimension)
models in mathematics, physics, chemistry, and engineering (e.g.,
nanotechnology and microelectronics), when one considers propagation of
waves of various nature (electromagnetic, acoustic, etc.) through a quasione-dimensional
system (often a mesoscopic one) that looks like a thin
neighborhood of a graph.  We refer to the survey \cite{MR2459876} and references therein.

{Let  $\Gamma$ be a metric graph 
$
\Gamma=\Gamma_f\cup \Gamma_\infty
$
where $\Gamma_f$ is made from the finite length edges of the graph $\Gamma$ whereas $\Gamma_\infty$ collects all the infinite edges.}
First of all in this paper, we introduce  $\Delta_{\Gamma}$, the Laplace operator  on a metric graph $\Gamma$. Associated to the Laplacian we
have an initial valued problem;  the classical heat equation on $\Gamma$. This problem is well-posed  and its solutions decay as time goes to infinity  in a similar way as the solutions of the  Cauchy problem
on the whole line  \cite{haeseler},  \cite{MR3985968}, \cite{MR2305092}.
We prove that the asymptotic behaviour (in terms of the existence of an asymptotic profile) of those solutions is comparable to the solution of the classical heat equation in the half line with Neumann boundary conditions at $x=0$ and initial data at $t=0$ given by a multiple of a Dirac mass at $x=0$.

%

\begin{theorem}\label{as.behaviour.intro.99}
	Let   $\bu$ the solution to the heat equation in $\Gamma$ with {integrable initial datum}. Then, for any $1\leq p\leq \infty$,
	\begin{equation}
\label{limit.1.intro.99}
  t^{\frac12(1-\frac1p)}\|\bu(t)-\bU_M(t)\|_{L^p(\Gamma_{\infty})}\rightarrow 0, \qquad \text{as}\ t\rightarrow \infty,
\end{equation}
and
\begin{equation}
\label{limit.2.99}
 t^{\frac12 }  \| \bu(t)-\bU_M(t) \|_{L^p(\Gamma_{f})}\rightarrow 0, \qquad \text{as}\ t\rightarrow \infty,
\end{equation}
where $M$ is the total mass of the initial datum    and
\begin{equation}
\label{profil.w.intro.99}
  \bU_M (x,t)=\frac {2M}Nt^{-\frac12}\left\{
  \begin{array}{ll}
  	G(x/\sqrt t), &x\in \Gamma_{\infty},\\[5pt]
  	G(0),& x\in \Gamma_{f}.
  \end{array}
  \right.
\end{equation}
Here $G(s)$ is given by the classical Gaussian profile,
$$
G(s) = \frac1{\sqrt{4\pi}}e^{-\frac{s^2}{4}}.
$$
\end{theorem}

Notice that here the classical Gaussian profile appears as $\bU_M (x,t)$ is the
solution to the heat equation in a half-line with Neumann boundary conditions.

Now, let us turn our attention to nonlocal diffusion equations with a convolution kernel.
Equations of the form
\begin{equation} 
\label{a11}
u_t (x,t) = (J*u-u) (x,t) = \int_{\mathbb{R}} J(x-y)u(y,t) \, dy - u(x,t),
\end{equation}
and variations of it, have been recently widely used to model
diffusion processes. Here, and in what follows, $J:\mathbb{R} \mapsto \mathbb{R}_+$
belongs to $L^1(\rr) \cap L^1(\rr,|x|^2) $ and is assumed to be symmetric,
$J(z)$ is non-increasing for $z>0$, nonnegative in $\rr$ and strictly positive in neighbourhood of the origin.
To simplify a little the statements we will also assume that $A=1$ where
{\[
A:=\frac 12 \int_{\rr} z^2J(z)dz<\infty
\]}
and comment
on the general case when appropriate.  {Notice that, when $\varphi$ is a smooth function, a simple Taylor expansion gives
$$
\eps^{-3}\int _{I} J\Big(\frac{|x-y|}\eps \Big) (\varphi(y)-\varphi(x))dy \approx \left(\frac12 \int_{\rr} J(z) z^2 dx \right) \varphi_{xx} (x)
$$
for $x$ an interior point of an interval $I$. As a consequence for a general $A$ the limit profiles should be rescaled correspondingly.}

When $J$ has unit integral, as stated in \cite{F}, if
$u(x,t)$ is thought of as a density at the point $x$ at time $t$
and $J(x-y)$ is thought of as the probability distribution of
jumping from location $y$ to location $x$, then $\int_{\mathbb{R}}
J(y-x)u(y,t)\, dy = (J*u)(x,t)$ is the rate at which individuals
are arriving at position $x$ from all other places and $-u(x,t) =
-\int_{\mathbb{R}} J(y-x)u(x,t)\, dy$ is the rate at which they are
leaving location $x$ to travel to all other sites. This
consideration, in the absence of external or internal sources,
leads immediately to  the fact that the density $u$ satisfies
equation \eqref{a11}. These kind of nonlocal equations are used to model very different
applied situations, for example in biology (see \cite{CF},
\cite{MogEdel}), image processing (\cite{KOJ}, \cite{GO}),
particle systems (\cite{BodVel}), coagulation models
(\cite{FourLau}), etc.

Our next goal is to study nonlocal evolution problems with a convolution kernel on the metric graph $\Gamma$, that is,
we deal with
\begin{equation}\label{eq.1.intro.99}
\left\{
\begin{array}{ll}
\displaystyle \bu_t(t,x)=\int _{\Gamma} J(d(x,y)) (\bu(t,y)-\bu(t,x))dy, & x\in \Gamma , t > 0,\\[10pt]
\bu(0,x)=\bu_0(x),&x\in \Gamma.
\end{array}
\right.
\end{equation}
Here $d(x,y)$ is the natural distance in $\Gamma$, see Section \ref{sect-nota}.
Now, particles can jump inside the graph from position $x$ to position $y$ according to the probability kernel $J(d(x,y))$
(an interpretation similar to the one previously given also holds here).
Notice that particles may jump from one edge to another when the distance between points in the edges makes $J(d(x,y))$
positive.

For general nonlocal problems of this kind we refer to \cite{silvia} that contains closely related results concerning existence and uniqueness
of solutions and for different problems of this kind in the Euclidean setting to \cite{Cazacu,ChChR,CEQW,CEQW2,GMQ,IgnatIgnat,IR,IR2} and the book \cite{ellibro}. For the decay rate and the asymptotic profile in the real line (that coincides with the ones for the local heat
equation) we refer to \cite{ChChR}.
For a class of kernels that give exponential decay we quote \cite{IRSA}.

Here, we show that nonlocal problems are closely related to the local heat equation. In fact, one can
obtain solutions to the heat equation by approximating them by solutions to nonlocal problems.
To see this, for a rescaling parameter that acts in the kernel, $\varepsilon>0$, considering
$J_\eps (d(x,y)) = \eps^{-3}J (d(x,y) / \eps )$, we deal with a relaxation limit for
this nonlocal evolution problem and obtain that its solutions converge to the solutions of the local one
 when $\varepsilon$ goes to zero.

\begin{theorem}
	\label{th.relax.intro.99}
It holds that the solutions to the rescaled nonlocal problem with a fixed initial datum verify
$$
\bu^\eps \rightharpoonup \bu \ \text{ weakly in} \ L^2([0,\infty), L^2(\Gamma)).
$$
where  $\bu$ is the unique solution to the heat equation in $\Gamma$
with the same initial condition.
\end{theorem}

Also we obtain the decay of the solutions to the nonlocal problem on metric graphs, and realize that the asymptotic  behaviour of those solutions is comparable to the solution of a related problem with the classical heat equation that we studied first.

\begin{theorem}
	\label{th.nonlocal.1.intro.99}
Let $1\leq p<\infty$ and  $\bu_0\in L^1(\Gamma)\cap  L^p(\Gamma) $.  
For any $1\leq q<p$ the  solution
	to \eqref{eq.1.intro.99}  satisfies
	\begin{equation}
\label{limit.p.intro.99}
t^{\frac 12(1-\frac 1q)}  \|\bu(t)-\bU_{M}(t)\|_{L^q(\Gamma)} \rightarrow 0, \ \text{as}\ t\rightarrow \infty,
\end{equation}
where the asymptotic profile $\bU_M$ is given by \eqref{profil.w.intro.99}.
\end{theorem}


We end this introduction with a very brief description of the methods used to prove our results.
Notice that here we can not use the Fourier transform as in \cite{ChChR,MR2356418} instead for the proof of the asymptotic behaviours,
both for the local and the nonlocal problems, we have to rely on energy and rescaling methods,
see \cite{MR2542582,IgnatIgnat} and also \cite{CM}. This scaling argument is delicate since it changes the graph in which the
rescaled solution is defined. When we pass to the limit what remains is the unbounded part of the graph,
while the finite edges collapse to a single point. This is the main reason why the asymptotic behaviour
(both for the local and nonlocal cases) is given by the Gaussian profile in a half line with the total mass
equally distributed according to the number of edges of infinite length.
For the relaxation limit we use ideas from \cite{ellibro,IgnatIgnat} where a compactness result is proved (see the Appendix
at the and of this paper).

{Related works concerning shrinking of a finite part of the graph can be found in \cite{Ber,Cacc}.
Remark that in our rescaling procedure every edge of the graph is rescaled, while in the previously mentioned references 
only some edges are changed. }

\medskip

The paper is organized as follows: first, in Section \ref{sect-nota} we introduce some notations and include the 
definition of the Laplcian in the graph;
in Section \ref{mainresults} we collect more precise statements of the main results of this paper (making clear the meaning of $\sim$
and the hypothesis on the initial datum
in Theorems \ref{as.behaviour.intro.99} and \ref{th.nonlocal.1.intro.99}). The proofs are postponed to the next two sections. Section
\ref{sect-local} is devoted to the study of the local problem (the heat equation), and in Section \ref{nonlocal-problem} we tackle
the nonlocal problem proving the relaxation limit and the asymptotic behaviour.  Finally, Section \ref{sect-Appendix} is an appendix where we include a compactness lemma that is needed when studying the relaxation
limit for the nonlocal equation.

\section{Notations and basic definitions} \label{sect-nota}

Before we present our results, we need to introduce some notations and basic definitions on metric graphs that we will use along the paper.
Let $\Gamma=(V,E)$ be a graph where $V$ is a set of vertices and $E$ the set of edges.
For each $v\in V$ we denote  $E_v=\{e\in E: v\in e\}$.
We assume that $\Gamma$ is a  finite connected graph. The edges could be of finite length and then their ends are vertices of $V$ or they have infinite length and then we assume that each infinite edge is a ray with a single vertex belonging to $V$ (see e.g. \cite{MR2459876} for more details on graphs with infinite edges). Here we assume that we do not have any terminal vertices, that is, vertices with degree one, and if we have a vertex with degree two, then we just remove it from our graph. Therefore,
we will always assume that the degree of each vertex is greater or equal than three.

We fix an orientation of $\Gamma$, that is, each edge $e$ is oriented.
Given an edge $e$, we denote by $I(e)$ the initial vertex and by $T(e)$ the terminal one.
In the case of infinite edges we have only initial vertices.
We identify every edge $e$ of $\Gamma$ with an interval $I_e$, where $I_e=[0,l_e]$ if the edge is finite and $I_e=[0,\infty)$ if the edge is infinite. This identification introduces a coordinate $x_e$ along the edge $e$. With this in mind we define a metric in $\Gamma$,
$d(x,y)$ stands for the distance between two points in the graph, that is given by the length of the minimal path that joins them.
In case both $x$ and $y$ belong to the same edge we have $d(x,y) = |x-y|$, when they belong to adjacent edges,
$d(x,y) = |x-e|+|y-e|$ with $e$ the vertex that is common to the two adjacent edges, when the minimal path involve three edges we have
$d(x,y) = |x-e_1|+|e_1-e_2| + |y-e_2|$, etc.  In the special case when the graph is star shaped (only one vertex and $N$ edges), we have
\[
d(x,y)=\left\{
\begin{array}{ll}
|x-y|, & x,y \in I_k, k=1,\dots, N,\\[5pt]
|x|+|y|, & x\in I_k, y\in I_j, k\neq j.
\end{array}
\right.
\]

In this way $\Gamma$ is a metric space that is often named as a metric graph, see e.g. \cite{MR2459876}. Moreover, we can write
\[
\Gamma=\Gamma_f\cup \Gamma_\infty
\]
where $\Gamma_f$ is made from the finite length edges of the graph $\Gamma$ whereas $\Gamma_\infty$ collects all the infinite edges.
In the special case in which there is only one infinite edge $\Gamma_\infty$ is just a half line.
Let $v$ be a vertex of $V$ and $e$ be an edge in  $E_v$. We set for finite edges $e$
$$j(v,e)=\left\{
\begin{array}{lll}
0& \text{if} &v=I(e), \\[5pt]
l_e& \text{if} & v=T(e)
\end{array}
\right.
$$
and
$$j(v,e)=0,\ \text{if}\ v=I(e)$$
for infinite edges.

We identify any function $\bu$ on $\Gamma$ with a collection $\{\ue\}_{e\in E}$ of functions $\ue$ defined on the edges  $e$ of $\Gamma$. Each $\ue$ can be considered as a function on the interval $I_e$.
We will use the notation $\bu$ to denote a function in the whole graph $\Gamma$ and $\ue$ to denote the function on the edge $e$.
In fact, with some abuse in the notation, we will use the same notation $\ue$ for both the function on the edge $e$ and the function on the interval $I_e$ identified with $e$.
For a function $\bu:\Gamma\rightarrow \rr$,  $\bu=\{u^e\}_{e\in E}$, and $f:\rr \mapsto \rr$ we denote by $f(\bu):\Gamma\rightarrow \rr$ the family
$\{f(u^e)\}_{e\in E}$, where  $f(u^e):e\rightarrow\rr$ is given by $f(u^e) (x_e) = f(u^e(x_e))$.

We introduce  $C(\Gamma)$ the space of continuous functions on $\Gamma$.
A function $\bu=\{\ue\}_{e\in E}$  is continuous on $\Gamma$ if and only if $\ue$ is continuous on $I_e$ for every $e\in E$, and we have continuity at the vertices, that is,
$$\ue(j(v,e))=u^{e'}(j(v,e')), \quad \forall \ e,e'\in E_v.$$

The space $\LL^p(\Gamma)$, $1\leq p<\infty$ consists of all functions   $\bu=\{u_e\}_{e\in E}$ on $\Gamma$ that belong to $\LL^p(I_e)$
for each edge $e\in E$ and we endow it with the norm
$$\|\bu\|_{\LL^p(\Gamma)}^p=\sum _{e\in E}\|u^e\|_{\LL^p(I_e)}^p.$$
Similarly, the space $\LL^\infty(\Gamma)$ consists of all functions that belong to $\LL^\infty(I_e)$ for each edge $e\in E$ and
$$\|\bu \|_{\LL^\infty(\Gamma)}=\max_{e\in E}\|u^e\|_{\LL^\infty(I_e)}.$$

The Sobolev space $H^1(\Gamma)$, consists in all   functions on $\Gamma$ that  belong to
$H^1(I_e)$ for each $e\in E$ and
$$\|\bu \|_{H^1(\Gamma)}^2=\sum _{e\in E}\|u^e\|_{H^1(e)}^2.$$
Sobolev spaces with higher differentiability $H^m(\Gamma)$, $m\geq 2$, can be defined in an analogous way.
The spaces $L^2(\Gamma)$ and $H^1(\Gamma)$ are Hilbert spaces with the inner products
$$(\bu,\bv)_{\LL^2(\Gamma)}=\sum _{e\in E}(\ue,\ve)_{\LL^2(I_e)}=\sum _{e\in E}\int _{I_e}\ue(x_e){\ve}(x_e)dx_e$$
and
$$(\bu,\bv)_{H^1(\Gamma)}=\sum _{e\in E}(\ue,\ve)_{H^1(I_e)}= \sum _{e\in E} \int _{I_e} \frac{d\ue}{dx} (x_e)  \frac{d\ve}{dx} (x_e) dx_e.$$
Notice that this defines an inner product in $H^1(\Gamma)$.

We now define the exterior normal derivative of a function $\bu=\{u^e\}_{e\in E}$  at the endpoints of the edges.
For each $e\in E$ and $v$ an endpoint of $e$ we consider the normal derivative of the restriction of $\bu$ to the edge $e$ of $E_v$ evaluated at $j(v,e)$ to be defined by:
$$\frac {\partial u^e}{\partial n_e}(j(v,e))=
\left\{
\begin{array}{lll}
-u_x^e(0+)&\text{if}& j(v,e)=0, \\[5pt]
u_x^e(l_e-)& \text{if}&j(v,e)=l_e  .
\end{array}
\right.
$$

We now introduce the Laplace operator $\Delta_\Gamma$ on the graph $\Gamma$. This is a standard procedure and we refer the interested reader to  \cite{MR1476363}. The operator  $\Delta_\Gamma$ has domain$$D(\Delta_\Gamma)=\Big\{\bu=\{u^e\}_{e\in E}\in H^2(\Gamma): \bu \ \text{continuous}\ \text{and}\  \sum _{e\in E_v} \frac {\partial u^e}{\partial n_e}(j(v,e))=0\quad
\text{for all}\ v\in V\Big\}$$
and it applies to any function $\bu\in D(\Delta_\Gamma)$ as follows
$$(\Delta_\Gamma \bu)^e=(u^e)_{xx}\quad \text{for all}\ e\in E .$$
In other words,  $D(\Delta_\Gamma)$ is the space of all continuous functions on $\Gamma$, $\bu=\{u^e\}_{e\in E}$,  such that for every edge $e\in E$,
$u^e\in H^2(I_e)$,  and the following Kirchhoff-type condition is satisfied,
$$\sum _{e\in E: T(e)=v} u^e_x(l_e-)-\sum _{e\in E: I(e)=v}u_x^e(0+)=0 \quad \text{for all} \ v\in V.$$
In the particular case when $\Gamma$ consists on only one edge these conditions reduce to the classical Laplacian with Neumann boundary conditions.

The quadratic form associated to $\Delta_\Gamma$ is given by
 $$\mathcal{Q}_{\Gamma}(\bu,\bu)=(\bu_x,\bu_x)_{L^2(\Gamma)}$$ for all $\bu\in D(\mathcal{Q}_\Gamma)$ where
\[
D(\mathcal{Q}_\Gamma)=D((-\Delta_\Gamma)^{1/2})=\Big\{\bu=\{u^e\}_{e\in E}\in H^1(\Gamma): \bu \ \text{continuous} \Big\}.
\]
In particular $D(\mathcal{Q}_\Gamma)$ with the $H^1(\Gamma)$-norm is a Banach space and
$\mathcal{Q}_\Gamma(\bu,\bu)=-(\bu,\Delta_\Gamma \bu)$ for all $\bu\in D(\Delta_\Gamma)$.
For other kinds of Laplace operators on metric graphs, we refer to \cite{MR3013208,MR2277618}.

It is easy to verify that $(\Delta_\Gamma, D(\Delta_\Gamma ) )$ is a linear, unbounded, self-adjoint, dissipative operator on $\LL^2(\Gamma)$, i.e.
$(\Delta_\Gamma\bu,\bu)_{\LL^2(\Gamma)}\leq 0$ for all $\bu\in D(\Delta_\Gamma)$.
Since $
\bigoplus_{e\in \Gamma} C_c^\infty(I_e)\subset D(\Delta_\Gamma)$
we obtain that $D(\Delta_\Gamma)$ is dense in any $L^p(\Gamma)$, $1\leq p<\infty$. The analysis of the operator $\Delta_\Gamma$ on $L^p(\Gamma)$-spaces will be discussed in Section \ref{sect-local}.

The main results of this paper are proved by rescaling the solutions with a parameter that we will call $\lambda$,
considering $\bu_\lambda(x)=\lambda\bu(\lambda x)$.
Since the equation is defined on edges that are either finite or infinite we have to rescale properly not only the solutions of our equations but also their domain of definition $\Gamma^\lambda=(V,E^\lambda)$. The graph $\Gamma^\lambda$ is obtained from $
\Gamma=(V,E)$ by rescaling properly the intervals $I_e^\lambda$ that parametrize the edges $E^\lambda$ of the graph $\Gamma^\lambda$.
We set
\[
I_e^\lambda=
\left\{
\begin{array}{ll}
	[0,l_e/\lambda], & \text{if}\ l_e<\infty, \\[5pt]
	[0,\infty),& \text{if}\ l_e=\infty.
\end{array}
\right.
\]
In this way we have $\Gamma^\lambda=\Gamma_f^\lambda\cup \Gamma_\infty$ since $\Gamma_\infty$ remains invariant under the above transformation.

\section{Main results}\label{mainresults}

In this section we collect the statements of the main results of this manuscript. We postpone their proofs to the next sections.  We also write related literature and motivation to research in local and nonlocal evolution problems on metric graphs.

\subsection{\bf The local evolution problem.}
Associated with the Laplacian we consider the heat equation on $\Gamma$:
\begin{equation}\label{eq.tree.intro}
\left\{
\begin{array}{ll}
\bu _t(t,x)-\Delta_\Gamma \bu(t,x)=0,& x\in \Gamma, t > 0 ,\\[5pt]
\bu(0,x)=\bu_0 (x) ,&   x\in \Gamma.
\end{array}
\right.
\end{equation}
We point out that the  well-posedness
for this problem is a direct consequence of Hille-Yosida-Phillips theorem. For example  if the initial data $\bu_0$ belongs to $L^1(\Gamma)$ then there exists a unique solution $\bu\in C([0,\infty),L^1(\Gamma))$ and it satisfies
	\begin{equation}
		\label{est.1.intro}
		\int _{\Gamma}\bu(t,x)dx=\int _{\Gamma} \bu_0(x)dx,
	\qquad
		\|\bu(t)\|_{L^1(\Gamma)}\leq \|\bu_0\|_{L^1(\Gamma)}.
	\end{equation}	
Moreover, the ultracontractivity of the semigroup generated by the operator $\Delta_\Gamma$ gives us that the solution $\bu$ also decays when the time increases
	\begin{equation}
		\label{est.3.intro}
		\|\bu(t)\|_{L^p(\Gamma)}\leq C(p,\Gamma) t^{-\frac12(1-\frac 1p)}\|\bu_0\|_{L^1(\Gamma)}, \qquad 1\leq p\leq \infty.
	\end{equation}	
More details about the well-poesedness of problem \eqref{eq.tree.intro} will be given in  Section \ref{sect-local}.
%

Starting from these estimates and using a rescaling procedure we can prove that for a connected graph with a finite number of edges the long time behaviour of the solution is given by the same problem in the star-shaped tree that is obtained when the bounded segments, the compact core of the graph, collapse to a single point. In this case the asymptotic profile is given by the solution to the classical heat equation in the half line with Neumann boundary condition at $x=0$ and as initial datum at $t=0$ a Dirac mass at $x=0$. This last solution is self similar with a well known Gaussian profile.
To obtain the limit profile for our evolution problem, we just repeat this Gaussian profile in any infinite line and multiply by $2 M/N$, being $M$ the total mass, and $N$ the number of infinity length segments.

\begin{theorem}\label{as.behaviour.intro}
	Let $\bu_0\in L^1(\Gamma)$ and  $\bu$ the solution of the problem \eqref{eq.tree.intro}. Then for any $1\leq p\leq \infty$
	\begin{equation}
\label{limit.1.intro}
  t^{\frac12(1-\frac1p)}\|\bu(t)-\bU_M(t)\|_{L^p(\Gamma_{\infty})}\rightarrow 0, \qquad \text{as}\ t\rightarrow \infty,
\end{equation}
and
\begin{equation}
\label{limit.2}
 t^{\frac12 }  \| \bu(t)-\bU_M(t) \|_{L^p(\Gamma_{f})}\rightarrow 0, \qquad \text{as}\ t\rightarrow \infty,
\end{equation}
where $M$ is the total mass of the initial datum $\bu_0$ and
\begin{equation}
\label{profil.w.intro}
  \bU_M (x,t)=\frac {2M}Nt^{-\frac12}\left\{
  \begin{array}{ll}
  	G(x/\sqrt t), &x\in \Gamma_{\infty},\\[5pt]
  	G(0),& x\in \Gamma_{f}.
  \end{array}
  \right.
\end{equation}
\end{theorem}

\subsection{\bf The nonlocal evolution problem.}

Now, let us look at nonlocal equations on a metric graph $\Gamma$.
We consider the evolution problem,
\begin{equation}\label{eq.1.intro}
\left\{
\begin{array}{ll}
\displaystyle \bu_t(t,x)=\int _{\Gamma} J(d(x,y)) (\bu(t,y)-\bu(t,x))dy, & x\in \Gamma , t > 0,\\[10pt]
\bu(0,x)=\bu_0(x),&x\in \Gamma.
\end{array}
\right.
\end{equation}

The kernel $J\in L^1(\rr) \cap L^1(\rr,|x|^2) $ is assumed to be symmetric,
$J(z)$ is {non-increasing} for $z>0$, $J$ is  {nonnegative} in $\rr$ and positive in neighborhood of the origin.
Notice that we are not assuming the solution $\bu$ to the nonlocal evolution problem \eqref{eq.1.intro} to be
continuous, and hence, we do not impose any
condition on the vertices. 

{
The assumption that $J\in L^1(\rr) $ guarantees that operator $$
	\bu \mapsto L(\bu) := \int_{\Gamma}J(d(x,y))(\bu(y)-\bu(x))dy
	$$
	is bounded between any $L^p(\Gamma)$ spaces, $1\leq p \leq \infty$.
	It generates a contraction semigroup in these spaces which is positive preserving and Markovian.
 Results about the well posedness are given in Theorem \ref{th.nonlocal.1}.
}

Both local and nonlocal models are similar in the sense that they share some properties such as existence and uniqueness of solutions,
conservation of the total mass of the initial datum and
the validity of the strong maximum principle.
However, solutions to the nonlocal problem do not have a regularizing effect in time. Solutions are as smooth in space for $t>0$
as the initial data are (this fact is due to the integrability of the kernel $J$).
If $\bu_0$ belongs to a suitable space $X$, then $\bu(t) \in X$
for all times and it is no better (there is no regularizing effect in strong contrast with the local heat equation).
This lack of smoothing is a major difficulty for the analysis of the asymptotic behaviour
of solutions since it implies some lack of compactness of the trajectories, $\{\bu(t)\}_{t>0}$. {Since there is a lack of regularizing effect we cannot analyze the ultracontractivity property in the classical sense (\cite{MR1103113}): for initial data in $L^1(\Gamma)\setminus L^\infty(\Gamma)$ the semigroup is not in $L^\infty(\Gamma)$ at any positive time. However, we can prove that for initial data in $L^1(\Gamma)\cap L^p(\Gamma)$ the solution has a certain decay in the $L^p(\Gamma)$ norm, see Theorem \ref{th.nonlocal.decay.1}.
}

\subsubsection{\bf A relaxation limit.}

First, we establish one more evidence to support that the non-local problem \eqref{eq.1.intro}
and the local heat equation \eqref{eq.tree.intro} are closely related.
To this end
 let us consider the following relaxation problem: for each $\eps>0$ consider the problem
\begin{equation}\label{eq.eps}
\left\{
\begin{array}{ll}
\displaystyle \bu_t^\eps(t,x)=\eps^{-3}\int _{\Gamma} J \Big(\frac{d(x,y)}\eps \Big) (\bu^\eps(t,y)-\bu^\eps(t,x))dy, & x\in \Gamma, t>0,\\[10pt]
\bu^\eps(0,x)=\bu_0(x),&x\in \Gamma.
\end{array}
\right.
\end{equation}
and analyze the limit problem when $\eps\rightarrow 0$. We recall that to simplify the presentation we assume that
	\begin{equation}
\label{second.momentum.J}
  \frac 12 \int_{\rr} z^2J(z)dz=1.
\end{equation}

\begin{theorem}
	\label{th.relax.intro}
		For any $\bu_0\in L^2(\Gamma)  $
	 it holds that
$$
\bu^\eps \rightharpoonup \bu \ \text{ weakly in} \ L^2([0,\infty), L^2(\Gamma)),
$$
where  $\bu\in C([0,\infty),L^2(\Gamma))$ is the unique solution of the heat equation \eqref{eq.tree.intro}
with the same initial condition.
\end{theorem}

Notice that, when $u$ is a smooth function, a simple Taylor expansion gives
$$
\eps^{-3}\int _{I} J\Big(\frac{|x-y|}\eps \Big) (u(y)-u(x))dy \approx \left(\frac12 \int_{\rr} J(z) z^2 dx \right) u_{xx} (x)
$$
for $x$ an interior point of an interval $I$. What is remarkable here is that in the limit as $\eps \to 0$ we recover the Kirchoff
conditions on the nodes without assuming any condition on $\bu^\eps$ (not even continuity). For other relaxation limits
of this kind we refer to \cite{CERW,MR} and the book \cite{ellibro} and references therein.

It will be interesting to analyze under which conditions on function $J$ we can recover in the relaxation limit more general operators $\Delta_{A,B}$ 
like the ones described in  \cite{MR2277618}, \cite{MR3013208}. These acts exactly as the laplacian on the edges but there are different coupling conditions at the vertices of the type
$A\bu'(v)+B\bu(v)=0$ with suitable conditions under the matrices $A$ and $B$. This will be investigated in a future work.

\subsubsection{\bf Asymptotic behaviour for the nonlocal problem.}
In spite of this lack of regularizing effect, in the one dimensional case, when the problem is posed in the whole $\mathbb{R}$,
is it shown in \cite{ChChR} that solutions to the equation
$$
u_t(t,x)=\int_{\mathbb{R}} J(x-y) (u(t,y)-u(t,x))dy
$$
have the same asymptotic behaviour as solutions to the classical heat equation $u_t= u_{xx}$.
Therefore, we expect that
solutions to \eqref{eq.1.intro} in the graph also have the same asymptotic behaviour
(the same decay bounds \eqref{est.3.intro} and asymptotic profile) as solutions to the heat equation, \eqref{eq.tree.intro}, that we analyzed first in Theorem \ref{as.behaviour.intro}.

\begin{theorem}
	\label{th.nonlocal.1.intro}
Let $1\leq p<\infty$.	For any $\bu_0\in L^1(\Gamma)\cap  L^p(\Gamma) $  the  solution
	to \eqref{eq.1.intro}  satisfies
\begin{equation}
\label{decay.nonlocal.lp.intro}
  \|\bu(t)\|_{L^p(\Gamma)}\leq \frac {C(\|\bu_0\|_{ L^1(\Gamma)},\|\bu_0\|_{ L^p(\Gamma)})}{(1+t)^{\frac 12(1-\frac 1p)}},\quad \forall \ t>0.
\end{equation}
For any $1\leq q<p$ the asymptotic profile $\bU_M$ is given by \eqref{profil.w.intro}:
	\begin{equation}
\label{limit.p.intro}
t^{\frac 12(1-\frac 1q)}  \|\bu(t)-\bU_{M}(t)\|_{L^q(\Gamma)} \rightarrow 0, \ \text{as}\ t\rightarrow \infty.
\end{equation}
\end{theorem}

\begin{remark} {\rm The asymptotic profile is the same as the one that we obtained for the local problem. If we
don't assume $\frac 12 \int_{\rr} z^2J(z)dz=1$, we just let
$0<A=\frac 12 \int_{\rr} z^2J(z)dz<+\infty$ and we obtain that the asymptotic profile $ \bU_{A,M}$ verifies $\bU_{t}=A\Delta_{\Gamma}\bU$ and $\bU(0)=M\delta_0$ so $\bU_{A,M}(t,x)=\bU_M(At,x)$.
The convergence of the compact part of the graph can be slightly improved in the following sense: for a fixed  $1<p^*<p$ the following holds  for any $1\leq q\leq p^*$:
\[
t^{\frac 12(1-\frac 1{p^*})}  \|\bu(t)-\bU_{M}(t)\|_{L^q(\Gamma_f)} \rightarrow 0, \ \text{as}\ t\rightarrow \infty.
  \]
}
\end{remark}

\section{Asymptotic behaviour for the heat equation} \label{sect-local}
\setcounter{equation}{0}

Let us consider the heat equation on  $\Gamma$:
\begin{equation}\label{eq.tree}
\left\{
\begin{array}{ll}
\bu _t(t,x)-\Delta_\Gamma \bu(t,x)=0,& x\in \Gamma, t> 0 ,\\[5pt]
\bu(0,x)=\bu_0(x),&   x\in \Gamma.
\end{array}
\right.
\end{equation}

In this section we show that for a general graph the long time behaviour of the solution is given by the problem in a star-shaped tree
that is obtained from $\Gamma$ when the bounded segments, the compact core of the graph, collapse to a single point.
In this case the asymptotic profile is given by the solution in the half line with Neumann boundary condition at $x=0$ and initial data a Dirac mass at $x=0$
multiplied by an adequate constant that takes into account the number of infinite edges and the mass of the initial condition
(we multiply by $2M/N$, being $M$ the total mass, and $N$ the number of infinit length segments).

{We now present few well know facts about the semigroup generated by the Laplacian on metric graphs.  For the sake of completeness 
we prefer to include here some details for the interested reader.}
In the $L^2$-setting, using the fact that $(\Delta_\Gamma, D(\Delta_\Gamma))$ is maximal dissipative  we easily obtain that it generates a strongly continuous semigroup of contractions in $L^2(\Gamma)$ that we denote by $S(t)$.
This means that for any $\bu_0\in D(\Delta_\Gamma)$ there exists a unique solution $\bu(t)=S(t)\bu_0$ of system \eqref{eq.tree} that satisfies
$\bu\in C([0,\infty),D(\Delta_\Gamma))\cap C^1([0,\infty),\LL^2(\Gamma)).$ Moreover since $\Delta_\Gamma$ is self adjoint we also have more regularity on the solution, in particular $\bu\in C^\infty((0,\infty),D(\Delta_\Gamma))$.
Thus, for any $\bu_0\in D(\Delta_\Gamma)$ system \eqref{eq.tree} can be written in an explicit way as follows:
\begin{equation}\label{system1exp}
\left\{
\begin{array}{ll}
u^e\in C([0,\infty),H^2(I_e))\cap C^1([0,\infty),L^2(I_e)),& e\in E,\\[5pt]
u^e_t(t,x)- u_{xx}^e(t,x)=0,& x\in I_e, t> 0, \\[5pt]
\text{for all} \ v\in V, \ u^{e}(t,j(v,e))=u^{e'}(t,j(v,e')),& \forall e,e'\in E_v, t> 0,\\[5pt]
\displaystyle\sum _{e\in E: T(e)=v} u^e_x(t,l_e-)-\sum _{e\in E: I(e)=v}u_x^e(t,0+)=0, & \text{for all} \ v\in V.
\end{array}
\right.
\end{equation}

In the context of the $L^p$-spaces, $1\leq p<\infty$ there are several papers on the subject.
We refer to recent paper \cite{MR3985968} where using a direct method it is proved that more general second order operators under more general conditions generate a strongly continuous semigroup of contractions in $L^p(\Gamma)$, $1\leq p<\infty$. In particular, $\Delta_\Gamma$ with Kirchhoff
conditions fulfill the hypotheses in the previously mentioned reference \cite{MR3985968}.
Other approaches are possible. For example, if one starts from the $L^2$-theory and then use interpolation theory
one can obtain similar results. We cite here \cite{MR2291812} and \cite{MR2305092} where  the authors treat graphs with {finitely many edges
 of finite length} but general couplings. Their method allow to first prove that the $L^2$-semigroup is positive and contractive in $L^\infty(\Gamma)$. The proof relays on the characterization in \cite{MR2124040} based on Kato-type inequalities, i.e.
\begin{equation}
\label{kato.infty}
  (\Delta_\Gamma \bu, (|\bu|-1)^+\sgn (\bu))\leq 0, \qquad \forall \ \bu\in D(\Delta_\Gamma).
\end{equation}
Observe that when $\|\bu_0\|_{L^\infty(\Gamma)}\leq 1$ and we try to show that the solution satisfies $\|\bu (t)\|_{L^\infty(\Gamma)}\leq 1$ for all positive time, one derivative in time of $\|(\bu(t)-1)^+\|_{L^2(\Gamma)}^2$ leads to
\[
\frac 12 \int_{\Gamma}[(\bu(t,x)-1)^+]^2 dx=\int_{\Gamma}
\bu_t(t,x)  (\bu(t,x)-1)^+ \sgn(\bu(t,x))dx= (\Delta_\Gamma \bu, (|\bu|-1)^+\sgn (\bu))
\]
and inequality \eqref{kato.infty} gives the desired result.

A classical argument (see \cite[p.~56]{MR2124040}) involving Riesz-Thorin interpolation theory shows that $S(t)=e^{t\Delta_\Gamma}$ can be extended to a strongly continuous
semigroup of contractions in $L^p(\Gamma)$ for each $2\leq p<\infty$. By duality $e^{t\Delta^*_\Gamma}$  is strongly continuous and contractive for $1<p\leq 2$. An approximation on compact sets shows that it is also in $L^1(\Gamma)$. This also follow arguing as in \cite[Section~7.2.1]{Arendt}.

A special attention has been paid to the $L^1-L^\infty$ estimates of the semigroup. In \cite{haeseler}, see also \cite{MR2291812} for compact graphs, it has been proved that
\begin{align*}
\label{nash-haeseler}
  \|\bu\|_{L^2(\Gamma)}\leq C \|\bu _x\|_{L^2(\Gamma)}^{1/3}
   \|\bu \|_{L^1(\Gamma)}^{2/3}
\end{align*}
holds for a graph having some infinite edges attached to the compact part of the graph. In view of \cite{MR1103113} these estimates allow to prove the ultracontractivity of the semigroup, i.e.
\begin{equation}
\label{l2-linfty}
\|e^{t\Delta_\Gamma}\varphi \|_{L^\infty(\Gamma)}\leq C t^{-1/4}
\|\varphi\|_{L^2(\Gamma)}, \qquad \forall \ \varphi\in L^2(\Gamma).
\end{equation}
Duality arguments give us a similar $L^1-L^2$ estimates, which together with the previous one show an $L^1-L^\infty$ estimate on the semigroup. This shows that for any $\varphi\in L^1(\Gamma)$ the solution $\bu$ belongs to
$C((0,\infty),L^p(\Gamma)) $ for any $1\leq p<\infty$. In particular we can use the $L^2$-theory for $t>t_0$ with $t_0>0$ arbitrary to obtain regularity properties for the solution and justify all the formal computations done in the rest of the paper.

Therefore, the following existence and uniqueness result holds.

\begin{theorem}\label{existence}
For any $\bu_0\in D(\Delta_\Gamma)$ there exists a unique solution $\bu(t)$ of system \eqref{eq.tree} that satisfies
$\bu\in C([0,\infty),D(\Delta_\Gamma))\cap C^1([0,\infty),\LL^2(\Gamma)).$
Moreover, for any $\bu_0\in \LL^p(\Gamma)$, $1\leq p<\infty$, there exists a unique solution $\bu\in C([0,\infty),\LL^p(\Gamma))$ that satisfies
$$\|\bu(t)\|_{\LL^p(\Gamma)}\leq \|\bu_0\|_{\LL^p(\Gamma)}\quad \text{for all}\ t\geq 0.$$
\end{theorem}

\begin{remark}
	{\rm For any  $\bu_0\in D(\Delta_\Gamma)\cap L^1(\Gamma)$, there exists a unique solution $\bu\in C([0,\infty),D(\Delta_\Gamma)\cap L^1(\Gamma))$.}
\end{remark}

In the next result we obtain some decay bounds for the solutions.
\begin{theorem}
	\label{estimates}
	For any $\bu_0\in L^1(\Gamma)$ the solution of system \eqref{eq.tree} satisfies
	\begin{equation}
		\label{est.2}
		\|\bu(t)\|_{L^1(\Gamma)}\leq \|\bu_0\|_{L^1(\Gamma)},
	\end{equation}	
	\begin{equation}
		\label{est.1}
		\int _{\Gamma}\bu(t,x)dx=\int _{\Gamma} \bu_0(x)dx,
	\end{equation}
		\begin{equation}
		\label{est.3}
		\|\bu(t)\|_{L^p(\Gamma)}\leq C(p,\Gamma) t^{-\frac 12(1-\frac 1p)}\|\bu_0\|_{L^1(\Gamma)}, \ 1\leq p\leq \infty,
	\end{equation}	
			\begin{equation}
		\label{est.4}
		\|\bu_x(t)\|_{L^2(\Gamma)}\leq C( \Gamma) t^{-\frac 32}\|\bu_0\|_{L^1(\Gamma)},
	\end{equation}	
	and an energy estimate
	\begin{equation}
\label{energy}
  \int_0^T\int _{\Gamma}\bu_x^2(t,x)dxdt=\|\bu(0)\|_{L^2(\Gamma)}^2-\|\bu(t)\|_{L^2(\Gamma)}^2.
\end{equation}
\end{theorem}

\begin{proof}
The first property is the contractivity in $L^1(\Gamma)$ of the semigroup discussed above. The mass conservation follows by considering initial data  $\bu_0\in D(\Delta_\Gamma)\cap L^1(\Gamma)$, proving the property for these solutions and finally using an approximation argument. Indeed when the solution is regular
	we integrate in space variable   equation \eqref{system1exp} 	to obtain that the mass is constant in time. {Indeed, since $\bu(t)\in D(\Delta_\Gamma)$ and ${\bf 1}\in D(\mathcal{Q}_\Gamma)$
	\[
\frac{d}{dt}\int _{\Gamma}\bu(t,x)dx=(\Delta_\Gamma \bu,1)=-\mathcal{Q}	(\bu_x,{\bf 1}_x)=0
	\]}
	
	The third property follows from \cite[Th. 2.]{haeseler}.

The estimate on the derivative of $\bu$ follows since $\Delta_\Gamma$ is self-adjoint (see \cite[Section~3.2, p.35]{MR1691574}). Indeed, for any $\bv_0\in L^2(\Gamma)$ we have that $S(t)$, the semigroup generated by  $\Delta_\Gamma$,   satisfies $S(t)\bv_0\in D(\Delta_\Gamma)$ for any $t>0$ and
\[
\|\Delta_\Gamma S(t)\bv_0\|_{L^2(\Gamma)}\leq t^{-1} \|\bv_0\|_{L^2(\Gamma)}.
\]
Using this property with $\bv_0=S(t)\bu_0$ and estimate \eqref{est.3}  we obtain
\[
\|\Delta_\Gamma S(2t)\bu_0\|_{L^2(\Gamma)}\leq t^{-1}\|S(t)\bu_0\|_{L^2(\Gamma)}\leq C(\Gamma)t^{-1}t^{-\frac{1}{4}}\|\bu_0\|_{L^1(\Gamma)}=C(\Gamma) t^{-\frac{5}{4}}\|\bu_0\|_{L^1(\Gamma)}.
\]
Now, since $S(t)\bu_0\in D(\Delta_\Gamma)$ for any $t>0$ we get
\[
  \langle \partial_x (S(t)\bu_0),\partial_x (S(t)\bu_0)\rangle=
\mathcal{Q}_\Gamma(S(t)\bu_0, S(t)\bu_0)=
 \langle  S(t)\bu_0 , \Delta_\Gamma (S(t)\bu_0)\rangle.
\]
Using the above inequalities on the semigroup we obtain
\[
\| \partial_x (S(t)\bu_0)\|^2_{L^2(\Gamma)}\leq \|S(t)\bu_0\|_{L^2(\Gamma)} \| \Delta_\Gamma (S(t)\bu_0\|_{L^2(\Gamma)} \leq
t^{-\frac14} \|\bu_0\|_{L^1(\Gamma)} t^{-\frac54}\|\bu_0\|_{L^1(\Gamma)},
\]
which gives us the desired estimate.
\end{proof}

Now we are ready to proceed with the proof
of our main result in this section, namely Theorem \ref{as.behaviour.intro}.
%
%
%
%

In order to prove this result we use the method of rescaling the solutions. For any solution $\bu$ of system \eqref{eq.tree} we introduce the family
$\bu_\lambda:
\Gamma^\lambda\rightarrow\rr$, $\lambda>0$:
\[
\bu_\lambda(t,x)=\lambda \bu(\lambda^2t, \lambda x), \quad x\in \Gamma^\lambda, t>0.
\]
It is immediate to see that $\bu_\lambda$ satisfies the system
\begin{equation}\label{eq.tree.la}
\left\{
\begin{array}{ll}
\bu_{\lambda,t}(t,x)-\Delta_{\Gamma^\lambda} \bu_\lambda(t,x)=0,& x\in \Gamma^\lambda, t\neq 0 ,\\[5pt]
\bu_\lambda(0,x)=\bu_{0\lambda} (x),&   x\in \Gamma^\lambda.
\end{array}
\right.
\end{equation}
If $\bu_0\in L^1(\Gamma)\cap D(\Delta_\Gamma)$ the rescaled solution $\bu_\lambda$ satisfies a similar equation as \eqref{system1exp} but on the  rescaled intervals $I_e^\lambda$ instead of $I_e$.

First, we need to obtain some bounds. From now on we will look for constants that are independent of $\lambda$.
We will use the standard notation $\lesssim$ to denote a less or equal bound in which the involved constant does not
depend on the relevant quantities.

\begin{lemma}
	\label{estimates.ulambda}
	For any $\bu_0\in L^1(\Gamma)$ the family $\{\bu_\lambda\}_{\lambda>0}$ satisfies the following uniform estimates
	\begin{equation}
		\label{est.1.la}
		\int _{\Gamma}\bu_\lambda (t,x)dx=\int _{\Gamma} \bu_0(x)dx,
	\end{equation}
		\begin{equation}
		\label{est.3.la}
		\|\bu_\lambda(t)\|_{L^p(\Gamma)}\leq C(p,\Gamma) t^{-\frac 12(1-\frac 1p)}\|\bu_0\|_{L^1(\Gamma)},
	\end{equation}	
	and
			\begin{equation}
		\label{est.4.la}
		\|\bu_{\lambda,x}(t)\|_{L^2(\Gamma)}\leq C(p,\Gamma) t^{-\frac 34}\|\bu_0\|_{L^1(\Gamma)}.
	\end{equation}	
	Moreover, for any $\tau >0$,
	\begin{equation}
\label{energy.lambda}
    \int_\tau^T\int _{\Gamma}\bu_{\lambda,x}^2(t,x)dxdt\leq
     C(\Gamma) \tau^{-\frac{1}{2}} \|\bu_0\|_{L^1(\Gamma)}^2.
\end{equation}
	Finally, there exists a positive constant $C$ such that  for any $\lambda>1$,
	\begin{equation}
\label{tail.control}
  \int _{\Gamma_\infty, |x|>2R}|\bu_\lambda(t,x)|dx\leq \int _{\Gamma_\infty,|x|>R}|\bu_0(x)|dx+\frac {CM t}{R^2}.
\end{equation}
\end{lemma}

\begin{proof}
	The first three estimates follow directly from Theorem \ref{estimates}.
	The forth bound follows by using the energy identity \eqref{energy} with initial data at time $\tau$
	 and the uniform decay of the $L^2$-norm of $\bu_\lambda$.
	
	We prove estimate \eqref{tail.control} for nonnegative solutions since the general case follows immediately. Let us consider $u_\lambda^1, \dots, u_\lambda^N $ the restriction of $\bu_\lambda$ to $\Gamma_\infty$. We will prove the desired estimate only for $u_\lambda^1$ since the others
	cases are similar. Summing the estimates for each $u_\lambda^k$ gives us the desired estimate.
	
	Let us consider a smooth   function $\psi:(0,\infty)\rightarrow\rr$ such that $0\leq \psi \leq 1$ and $\psi\equiv 0$ in $(0,1)$ and $\psi\equiv 1$ in $x\in (2,\infty)$.  We define $\psi_R(x)=\psi(x/R)$.
	Multiplying the equation satisfied by $u_\lambda^1$ with $\psi_R$ and integrating in space and time we obtain:
	\begin{align*}
\int _0^\infty u_\lambda^1 (t,x)\psi_R(x)dx&=\int_0^\infty u_\lambda^1(0,x)\psi_R(x)dx+\int _0^t \int _0^\infty u_\lambda^1(s,x)(\psi_R)_{xx} (x) dxds\\
&\leq  \int_0^\infty u_\lambda^1(0,x)\psi_R(x)dx+ \frac{Mt}{R^2} \|\psi''\|_{L^\infty((0,\infty))} \\
& \leq \int_{\lambda R}^\infty u^1(0,x)\psi_R(x/\lambda ) dx+ \frac{Mt}{R^2} \|\psi''\|_{L^\infty((0,\infty))}.
\end{align*}
Since $\psi_R\equiv 1$ on $(2R,\infty)$  and $\lambda>1$ we obtain
\[
\int _{2R}^\infty u_\lambda^1 (t,x) dx\leq  \int_{ R}^\infty u_0^1(x)+ \frac{CMt}{R^2},
\]
which finishes the proof.
\end{proof}

\begin{proof}[Proof of Theorem~\ref{as.behaviour.intro} ]
We divide the proof in several steps. We will prove the results for $\bu_0\in L^1(\Gamma)\cap D(\Delta_\Gamma)$ and then by an approximation argument we will obtain the result for $L^1$-initial data.

{\bf Step I. Compactness of the family $\bu_\lambda$.}
	We prove that up to a subsequence $\lambda_j \to \infty$ the family $\{ \bu_\lambda\}_{\lambda>0}$ converges to some function $\bU$.
		
	Let us choose  $0<\tau<T<\infty$. Using the estimate \eqref{est.4.la} in Lemma \ref{estimates.ulambda} we obtain that $\bu_\lambda$ is uniformly bounded in $L^\infty((\tau,T),H^1(\Gamma_\infty))$. Moreover each of its components on $\Gamma_\infty$,  $\partial_t \bu_\lambda|_{\Gamma_\infty}=(\partial_t  u_\lambda^1, \dots, \partial_t u_\lambda^N)$ are uniformly bounded in $L^2((\tau,T),H^{-1}(0,\infty)) $. These estimates and 
	Aubin-Lions compactness argument (see for example \cite{simon}) imply that each component is relatively compact in $C((\tau,T),L^2_{loc}((0,\infty))$ so	
	$(\bu_\lambda)_{\lambda>0}$ is relatively compact in $C((\tau,T),L^2_{loc}(\Gamma_\infty))$. Cantor's diagonal argument implies that, up to a subsequence, $(\bu_\lambda)_{\lambda>0}$ converges toward a function $\bU$  in $C((0,\infty),L^2_{loc}(\Gamma_\infty))$. In view of estimate \eqref{est.3.la} this convergence  implies in particular that for any positive time $t$ the sequence $\bu_\lambda(t)\rightharpoonup \bU(t)   $ in $L^p(\Gamma_\infty)$ for any $1<p<\infty$ and the bound in \eqref{est.3.la} transfers to $\bU$.
	
	 Moreover the above convergence in $C((\tau,T),L^2_{loc}(\Gamma_\infty))$ implies that the convergence also holds in $C((0,\infty),L^1_{loc}(\Gamma_\infty))$. The uniform tail control \eqref{tail.control} in Lemma \ref{estimates.ulambda} shows that the convergence also holds in $C((0,\infty),L^1(\Gamma_\infty))$. In particular, at time $t=1$, $\bu_\lambda (1)\rightarrow \bU(1)$ in $L^1(\Gamma_\infty)$ and this proves \eqref{limit.1.intro} when $p=1$. {Assume for the moment that we identified the profile $\bU$ in $\Gamma_\infty$ to be the one
	given by \eqref{profil.w.intro}, precisely $\bU_M$}. This means that all the family converges to this profile, not only a subsequence.
	Using that both $\bu$ and $\bU$ are uniformly bounded in $L^p(\Gamma)$, $1\leq p<\infty$, with a bound of order $t^{-1/2(1-1/p)}$ and the  convergence in $L^1(\Gamma_\infty)$ we deduce that \eqref{limit.1.intro} holds for all $p\in [1,\infty)$. {Moreover, since $\bU$ have been identified to be $\bU_M$ explicit computations give us  
	\[
	\|\bU_x(t)\|_{L^2(\Gamma_\infty)}\leq C(M,N) t^{-3/2}, \ \forall\ t>0.
	\]
	A similar estimates for $\bu$ obtained in Theorem \ref{estimates}  so to treat the case $p=\infty$  we can argue as follows
	\[
	\| \bu(t)-\bU(t)\|_{L^\infty(\Gamma_\infty)}\lesssim (\| \bu_x(t)\|_{L^2(\Gamma_\infty)}+\| \bU_x(t)\|_{L^2(\Gamma_\infty)})^{\frac12} \| \bu(t)-\bU(t)\|_{L^2(\Gamma_\infty)}^{\frac12}=o(t^{-\frac12}).
	\]
}

	\textbf{Step II. Identification of the limit.}
	Let us now consider a test function $$\varphi:C([0,\infty),H^1( \Gamma_\infty))\cap C^1([0,\infty),L^2(\Gamma_\infty)),$$ compactly supported in time and such that
	$\varphi^e(t,0)=\varphi^{e'}(t,0)$ for all $e,e'\in \Gamma_\infty$ (that is, we ask for continuity
	at the unique node of $\Gamma_\infty$). For each $t\geq 0$, we extend $\varphi(t)$ to the whole $\Gamma^\lambda$, to this end we take
	a function $\widetilde \varphi_\lambda:[0,\infty)\times \Gamma^\lambda\rightarrow \rr$ such that $\widetilde \varphi_\lambda$ is constant in the finite part of the graph $\Gamma_f^\lambda$, that is,
 \[
\widetilde\varphi_\lambda(t,x)=\left\{
\begin{array}{ll}
	\varphi(t,x) & \text{if}\ x\in \Gamma_\infty,\\[5pt]
	\varphi(t,0)  & \text{if}\ x\in \Gamma_f^\lambda,
\end{array}
\right.
\]
It follows that $\widetilde \varphi\in C([0,\infty), \mathcal{Q}(\Gamma^\lambda)$, i.e. it is not only in $H^1(\Gamma^\lambda)$ but it is also continuos.
Since $\bu_\lambda\in C([0,\infty), L^1(\Gamma)\cap D(\Delta_\Gamma))$ we multiply equation \eqref{eq.tree.la} by $\widetilde\varphi_\lambda$ and obtain
\begin{align*}
\label{}
 0&= \int_0^\infty \int_{\Gamma_\infty} \bu_\lambda \varphi_t-\int_0^\infty \int_{\Gamma_\infty}\bu_{\lambda,x}  \varphi_{x}
 \\
 & \qquad + \int_0^\infty \int_{\Gamma_f^\lambda} (\bu_\lambda \widetilde\varphi_{\lambda,t}-\bu_{\lambda,x} \widetilde\varphi_{\lambda,x})+\int_{\Gamma^\lambda} \bu_\lambda(0,x)\widetilde\varphi_\lambda(0,x)dx\\
  &=  \int_0^\infty \int_{\Gamma_\infty} \bu_\lambda \varphi_t-\int_0^\infty \int_{\Gamma_\infty}\bu_{\lambda,x}  \varphi_{x}+ \int_0^\infty \int_{\Gamma_f^\lambda} \bu_\lambda \widetilde \varphi_{\lambda, t}    +\int_{\Gamma^\lambda} \bu_\lambda(0,x)\widetilde\varphi_\lambda(0,x)dx\\
  &=\int_0^\infty \int_{\Gamma_\infty} \bu_\lambda \varphi_t-\int_0^\infty \int_{\Gamma_\infty}\bu_{\lambda,x}  \varphi_{x}+  \int_0^\infty \varphi_t(t,0) \int_{\Gamma_f^\lambda} \bu_\lambda +\int_{\Gamma^\lambda} \bu_\lambda(0,x)\widetilde\varphi_\lambda(0,x)dx\\
  &:=I_1^\lambda - I_2^\lambda+I_3^\lambda+I_4^\lambda,
\end{align*}
where we used that $\widetilde \varphi_\lambda$ is constant on $\Gamma_f^\lambda$.

Using the decay of the $L^2$-norm of $\bu_\lambda$ we obtain that $\bu_\lambda$ is uniformly bounded in $L^2((0,T),L^2(\Gamma_\infty))$:
\[
\int_0^T \|\bu_\lambda(t)\|^2_{L^2(\Gamma^\lambda)}dt\lesssim \int _0^T \frac 1{\sqrt t}dt \lesssim \sqrt T.
\]
	This means that up to a subsequence $\bu_\lambda\rightharpoonup \bU$ in $L^2((0,T),L^2(\Gamma_\infty))$.
Using that
 $\varphi$ has compact support in time we get
\begin{equation}
\label{lim.1}
  I_1^\lambda \rightarrow \int_0^\infty \int_{\Gamma_\infty} \bU (t,x)\varphi_t (t,x) dx dt .
\end{equation}
For the third term, $I_3^\lambda$, we use again the decay in the $L^2$-norm to obtain
\begin{align*}
\Big|\int_{\Gamma_f^\lambda} \bu_\lambda(t,x)  dx\Big| &\leq  \sum _{e\in \Gamma_f^\lambda} \int _0^{l_e/\lambda}|u_\lambda^e(t,x)|dx\leq
 \sum _{e\in \Gamma_f^\lambda} \|u_\lambda^e(t)\|_{L^2(e )}(l_e/\lambda)^{\frac12}\leq
  \frac{C(\Gamma)\|\bu_0\|_{L^1(\Gamma)}}{\lambda^{\frac12}t^{\frac14}} .
\end{align*}
Integrating in time and using that $\varphi$ has compact support in time we obtain that $I^\lambda_3\rightarrow 0$ as $\lambda\rightarrow\infty$.

 For the last term $I^\lambda_4$ we split it as follows
\begin{align*}
  I_4^\lambda &=\int_{\Gamma^\infty} \bu_\lambda(0,x)\widetilde\varphi_\lambda(0,x)dx+\int_{\Gamma_f^\lambda} \bu_\lambda(0,x)\widetilde\varphi_\lambda(0,x)dx\\
  &=\int_{\Gamma^\infty} \bu_\lambda (0,x)\varphi(0,x)dx
  +\int_{\Gamma_f^\lambda} \bu_\lambda(0,x)\varphi(0,0) dx\\
  &=\varphi(0,0) \int _{\Gamma^\lambda}\bu_{\lambda}(0,x)dx+\int_{\Gamma^\infty} \bu_\lambda (0,x)(\varphi(0,x)-\varphi(0,0))dx   \\
   &=\varphi(0,0) \int _{\Gamma }\bu_{0}(x)dx+\int_{\Gamma^\infty} \bu_0(x)\Big(\varphi(0,\frac x\lambda)-\varphi(0,0)\Big)dx \\
  & \qquad \rightarrow \varphi(0,0) \int _{\Gamma }\bu_{0}(x)dx, \quad \ \text{as}\ \lambda\rightarrow\infty.
\end{align*}

It remains to prove that
\[
I_2^\lambda\rightarrow \int_0^\infty \int_{\Gamma^\infty} \bU_x (t,x) \varphi_x (t,x) dx dt.
\]
First, we prove that for any $\tau>0$ the integrals in $(0,\tau)$ are small independently on $\lambda$. Indeed, using \eqref{est.4.la} we have
\[
\Big|\int_0^\tau \int_{\Gamma_\infty} \bu_{\lambda,x} (t,x) \varphi_x (t,x) dx dt  \Big|\leq \|\bu_{\lambda,x}\|_{L^1((0,\tau),L^2(\Gamma_\infty))}  \|\varphi_{x}\|_{L^\infty((0,\tau),L^2(\Gamma_\infty))} \lesssim \tau^{\frac14}.
\]

{We recall that at the begining of Step I we proved that $\bu_\lambda(t)\rightharpoonup \bU(t)$ in $L^p(\Gamma_\infty)$ for any $1<p<\infty$. Moreover, $(\bu_{\lambda,x}(t))_{\lambda>0} $ is uniformly bounded in $L^2(\Gamma_\infty)$. Hence, up to a subsequence $\bu_{\lambda,x}(t)\rightharpoonup \bU_x(t)$ in $L^2(\Gamma_\infty)$.}
Estimate \eqref{est.4.la} transfers to $\bU$ so  $\|\bU_x(t)\|_{L^2(\Gamma_\infty)}\lesssim t^{-\frac34} $ for all $t>0$ and a similar estimate holds for $\bU$
\[
\Big|\int _0^\tau \int _{\Gamma_\infty}\bU_x\varphi_xdx\Big| \lesssim \tau^{\frac{1}{4}}.
\]
In view of \eqref{energy.lambda} we get
\[
\int _{\tau}^\infty \int_{\Gamma_\infty} \bu_{\lambda,x} (t,x)\varphi_x (t,x) dx dt
\rightarrow \int _{\tau}^\infty \int_{\Gamma_\infty}  \bU_x (t,x) \varphi_x (t,x) dx dt,
\]
which proves the desired limit for $I_2^\lambda$.

In view of the above results it follows that the limit point $\bU$ satisfies
\[  \int_0^\infty \int_{\Gamma_\infty}\Big( \bU (t,x)\varphi_t(t,x)- \bU_x(t,x)\varphi_{x}(t,x) \Big) dx dt+M \varphi(0,0)=0.
\]

From Step I we know that the limit point $\bU$ belongs to $C((0,\infty), L^1(\Gamma_\infty))$ and moreover $\|\bU(t)\|_{L^1(\Gamma_\infty)}\leq \|\bu_0\|_{L^1(\Gamma_\infty)}$. Also the
 energy estimate \eqref{energy.lambda} shows that $\bu_\lambda$ is uniformly bounded in $L^2((\tau, T),D(\mathcal{Q}_{\Gamma_\infty}))$ hence the limit
 point $\bU$ belongs to $L^2((\tau,T),D(\mathcal{Q}_{\Gamma_\infty}))$ for any $\tau>0$. This implies that $\bU(t)\in D(\mathcal{Q}_{\Gamma_\infty})$ for a.e. $t>0$.

The continuity of $\bU$ at the common vertex of $\Gamma_\infty$ (or the Neumann boundary condition if $\Gamma_\infty$ has only one edge)
 guarantees  that for a test function $\varphi$ which is more regular, $\varphi\in C([0,\infty), D(\Delta_{\Gamma_\infty}))\cap C^1([0,\infty),L^2(\Gamma_\infty))$ and compactly supported in time, we can perform one more integration by parts and obtain that $\bU$ satisfies
\begin{equation}
\label{limit.U}
   \int_0^\infty \int_{\Gamma_\infty} \bU (t,x) \big(\varphi_t (t,x)+  \varphi_{xx}(t,x)\big) dx dt +M \varphi(0,0)=0.
\end{equation}
For the above equation we need only
$\bU\in L^1_{loc}((0,\infty),L^1(\Gamma_\infty))$  in order to apply the arguments in \cite{MR700049} and prove the uniqueness of the profile $\bU$ satisfying \eqref{limit.U}.

An easy computation shows that $\bU_M=(U^1,\dots, U^N)$ given by
\[
U_M^k(t,x)=\frac{2 M}{N} G_t(x), \quad k=1,\dots, N,
\]
where $$G_t(x) =
  	\frac1{\sqrt{4\pi t}}e^{-\frac{x^2}{4t}}$$ is the classical heat kernel,
verifies identity \eqref{limit.U}.
Assume that we have two solutions $\bU_1$ and $\bU_2$ of problem \eqref{limit.U} in $L^1_{loc}((0,\infty),L^1(\Gamma_\infty))$. Then $\bv=\bU_1-\bU_2\in L^1_{loc}((0,\infty),L^1(\Gamma_\infty))$ satisfies
\begin{equation}
\label{limit.v}
   \int_0^\infty \int_{\Gamma_\infty} \bv(\varphi_t+\varphi_{xx})=0,
\end{equation}
for all functions $\varphi\in C([0,\infty), D(\Delta_{\Gamma_\infty}))\cap C^1([0,\infty),L^2(\Gamma_\infty))$. We show what all the components $v_1,\dots, v_N$ of $\bv$ are equal and vanish identically. Let us fix a $T>0$ and a function $\psi\in C^{1,2}([0,T]\times [0,\infty))$ with $\psi(T)\equiv 0$, $\psi(0,t)=0$ for $t\in [0,T]$. Choosing $\varphi=(\psi,-\psi, 0, \dots, 0)$ we obtain that
\[
 \int_0^\infty \int_{\Gamma_\infty} (v_1-v_2)(\psi_t+\psi_{xx})=0.
\]
Using \cite[Lemma 3]{MR700049} we obtain that $v_1\equiv v_2$ in $[0,T]$. Similarly all the components of $\bv$ are identical.
Let us now choose a similar $\psi$ but now assume that $\psi_x(t,0)=0$ instead of the Dirichlet  boundary condition at $x=0$.
Choosing $\varphi=(\psi,\dots,\psi)$ in \eqref{limit.v} and using that all the components are identical we obtain
\[
\int_0^T\int_0^\infty v_1(\psi_t+\psi_{xx})=0.
\]
For any $f\in L^2((0,T),L^2(0,\infty))$ we solve the backward heat equation
(notice that here the "initial condition" is taken at $t=T$ and time is reversed)
\[
	\left\{
\begin{array}{ll}
	\psi_t+\psi_{xx}=f,& x\in (0,\infty), t\in (0,T),\\[5pt]
	\psi(T,x)=0, &x\in (0,\infty),\\[5pt]
	\psi_x(t,0)=0, & t\in (0,T),
\end{array}
	\right.
\]
and obtain a function that $\psi$ that can be used as test function in the integral identity satisfied by $v_1$. This shows that $v_1\equiv 0$ in $(0,T)\times (0,\infty)$ so $\bv\equiv 0 $ on $(0,T)\times\Gamma_\infty$. We conclude that every limit point of $\bu_\lambda$, $\bU$, is given by \eqref{limit.U} and therefore the results in Step I hold for the whole family $(u_\lambda)$ and not only for a subsequence.

\textbf{Step III. Convergence on the compact part of the graph.}

Let us consider $\bw(t)\in D(\mathcal{Q}_{\Gamma_\infty})$, defined as follows
\[
\bw(t)=\bu(t)-\bU_M(t).
\]
Since the graph $ \Gamma_f$ is finite it is sufficient to prove \eqref{limit.2} in the case  $p=\infty$.
Let us recall that since $\Gamma$ has at least one infinite edge we have for any $v\in D(\mathcal{Q}_{\Gamma_\infty})$ that
\[
\|v\|_{L^\infty(\Gamma)}\leq C(\Gamma)\|v_x\|_{L^2(\Gamma)}^{\frac{1}{2}} \|v\|_{L^2(\Gamma)}^{\frac{1}{2}}.
\]
{
We apply the above inequality to $\bw$
\[
\|\bw(t)\|_{L^\infty(\Gamma_f)}\leq \|\bw(t)\|_{L^\infty(\Gamma)}\leq C(\Gamma))\|\bw_x\|_{L^2(\Gamma)}^{\frac{1}{2}}  \|\bw(t)\|_{L^2(\Gamma)}^{\frac{1}{2}}.
\]
Observe that $\bw_x(t)=\bu_x(t)-\bU_{M,x}(t,x)1_{\Gamma_\infty}(x)$. Explicit computations on $\bU_M$  and  \eqref{est.4} gives us 
\[
\|\bw_x(t)\|_{L^2(\Gamma)} \leq \|\bu_x(t)\|_{L^2(\Gamma)} +\| \bU_{M,x}(t)\|_{L^2(\Gamma_\infty)}\leq
C(\Gamma)(\|\bu_0\|_{L^1(\Gamma)} +M )t^{-3/4}.
\]}
It is then sufficient to prove that $\|\bw(t)\|_{L^2(\Gamma)}=o(t^{-\frac{1}{4}})$ as $t\rightarrow \infty.$ This estimate on $\Gamma_\infty$ has been proved in Step I. For $\Gamma_f$ we use \eqref{est.3} for $p=\infty$ and the explicit form of $\bu_M$ on $\Gamma_f$:
\begin{align*}
\|\bw(t)\|_{L^2(\Gamma_f)}&\leq C(\Gamma)\|\bw(t)\|_{L^\infty(\Gamma_f)}\leq C(\Gamma)(\|\bu(t)\|_{L^\infty(\Gamma_f)}+\|\bU_M(t)\|_{L^\infty(\Gamma_f)})\\
&\leq C(\Gamma)t^{-\frac{1}{2}}=o(t^{-\frac{1}{4}}), \ t\rightarrow \infty.
\end{align*}
This complete the case of the finite graph and finishes the proof.
\end{proof}

\section{The nonlocal problem} \label{nonlocal-problem}
\setcounter{equation}{0}

In this section we consider the nonlocal problem in the metric graph $\Gamma$,
\begin{equation} \label{eq.1}
\left\{
\begin{array}{ll}
\displaystyle \bu_t(t,x)=\int _{\Gamma} J(d(x,y)) (\bu(t,y)-\bu(t,x))dy, & x\in \Gamma, t>0,\\[10pt]
\bu(0,x)=\bu_0(x),&x\in \Gamma.
\end{array}
\right.
\end{equation}
Here $d(x,y)$ stands for the distance between two points in the graph. As we mentioned in the introduction, this distance $d(x,y)$
is the length of the minimal path that joins $x$ and $y$.
The kernel $J\in L^1(\rr) \cap L^1(\rr,|x|^2) $ is assumed to be symmetric,
$J(z)$ is non-increasing for $z>0$, $J$ is nonnegative in $\rr$ and positive in neighbourhood of the origin.

Our first goal is to show existence and uniqueness of  solution for the problem \eqref{eq.1}.
\begin{theorem}\label{th.nonlocal.1}
	For any $\bu_0\in L^p(\Gamma) $, $1\leq p\leq \infty$, there exists an unique solution
	$\bu\in C([0,\infty),L^p(\Gamma))$ of system \eqref{eq.1} satisfying
	\begin{equation}
\label{est.nonlocal.lp}
\|\bu(t)\|_{L^p(\Gamma)}\leq \|\bu_0\|_{L^p(\Gamma)}.
\end{equation}
{Also, for nonnegative initial datum the solution remains nonnegative.
}

For any $\bu_0\in L^2(\Gamma)$ the following energy estimate holds
  \begin{equation}
\label{energy.nonlocal}
\int _{\Gamma} \bu^2(t,x)dx+ \int_0^t \int _{\Gamma}\int_{\Gamma} J(d(x,y)) (\bu(s,x)-\bu(s,y))^2dxdyds=\int _{\Gamma} \bu^2_0(x)dx.
\end{equation}
Moreover,
\begin{equation}
	\label{grad.estimate}
  \mathcal{E}^J_\Gamma(\bu(t),\bu(t)):=\int _{\Gamma}\int_{\Gamma} J(d(x,y)) (\bu(t,x)-\bu(t,y))^2dxdy\leq
 \frac{ \|\bu_0\|^2_{ L^2(\Gamma)}}{t}.
\end{equation}
\end{theorem}

\begin{proof}[Proof of Theorem \ref{th.nonlocal.1}]
	The existence and uniqueness follow easily since the operator
	$$
	\bu \mapsto L(\bu) := \int_{\Gamma}J(d(x,y))(\bu(y)-\bu(x))dy
	$$
	is bounded between any $L^p(\Gamma)$ spaces (we refer to \cite{silvia} for extra details). Therefore, problem \eqref{eq.1} has a unique strong solution
	$\bu \in C^\infty(\mathbb{R}, X)$, for any $X=L^p (\Gamma)$ given by
	$
\bu(t)=e^{Lt}\bu_0.
$
The mapping
$
t \in (0,\infty) \mapsto \bu(t)=e^{Lt}\bu_0 \in X
$
is analytic. Moreover, the mapping $(t,\bu_0) \mapsto e^{Lt}\bu_0$ is continuous and contractive, that is,
$\|\bu(t)\|_{L^p(\Gamma)}\leq \|\bu_0\|_{L^p(\Gamma)}$ holds. {Indeed, for any  function $\rho\in C^1(\rr)$ with 
$\rho'$ nondecreasing we have
\[
  \frac{d}{dt}\int_{\Gamma}\rho(\bu(t,x))dx=(\bu_t(t),\rho'(\bu(t)))=(L\bu(t),\rho'(\bu(t)))=
- \mathcal{E}^J_\Gamma(\bu(t),\rho'(\bu(t)))\leq 0.
\]
An approximation argument shows that for any convex  function $\rho$ the map 
$t\rightarrow \int_{\Gamma}\rho(\bu(t,x))dx $ is nonincreasing.
Particular cases $\rho(s)=|s|^p$, $1\leq p<\infty$ show the contractivity in the $L^p$-norms.
%
When $p=\infty$ we  consider $\rho(s)=(|s|-M)^+$ and $M=\|\bu_0\|_{L^\infty(\Gamma)}$. For $\rho(s)=s^+$ we obtain the positivity property of the semigroup.
}

	In particular we have that
\[
  \frac 12 \frac {d}{dt}\int _{\Gamma} \bu^2(t,x)dx=-\frac 12 \int _{\Gamma}\int_{\Gamma} J(d(x,y)) (\bu(t,x)-\bu(t,y))^2dxdy.
\]
Estimate \eqref{grad.estimate} is classical for any self-adjoint operator $L$ satisfying $(Lu,u)\leq 0$,  see for example \cite[Th.3.2.1]{MR1691574}
which finishes the proof.
\end{proof}

\subsection{Relaxation limit}

Let us now consider the following relaxation problem: for each $\eps>0$ consider the system
\begin{equation}\label{eq.eps.77}
\left\{
\begin{array}{ll}
\displaystyle \bu_t^\eps(t,x)=\eps^{-3}\int _{\Gamma} J\Big(\frac{d(x,y)}\eps\Big) (\bu^\eps(t,y)-\bu^\eps(t,x))dy, & x\in \Gamma, t>0,\\[10pt]
\bu^\eps (0,x)=u_0(x),&x\in \Gamma.
\end{array}
\right.
\end{equation}
and analyze the limit problem when $\eps\rightarrow 0$. Let us mention that here we fix the initial datum in contrast with the analysis of the first term in the asymptotic behavior of the solutions by self-similarity where we also have to rescale the initial data.
Here we use compactness arguments instead of scaling ones (the kernel of the nonlocal operator is
rescaled with $\varepsilon$, but the spatial domain $\Gamma$ in which solutions are defined is unchanged).
Recall that we assumed
	\begin{equation}
  \frac 12 \int_{\rr} z^2J(z)dz=1.
\end{equation}

\begin{proof}[Proof of Theorem \ref{th.relax.intro}]
Let us remark that we have the following energy estimate
\begin{equation}
\label{energy.relax}
  \int _{\Gamma} (\bu^\eps)^2(t,x)dx +\eps^{-3} \int_0^T \int_{\Gamma} \int_{\Gamma} J\Big(\frac{d(x,y)}\eps \Big) (\bu^\eps(t,y)-\bu^\eps(t,x))^2
dxdydt =
\int_{\Gamma} \bu_0^2(x)dx.
\end{equation}
We will use the above identity in three different ways, by taking the same edge, two adjacent edges or two edges that does not have
a vertex in common.

As before, we divide our arguments into several steps.

{\bf Step I.}
We observe that $\bu^\eps$ is uniformly bounded in $C([0,T],L^2(\Gamma))$, so in particular in $L^2((0,T),L^2(\Gamma))$. So there exists $\bU=(U_e)_{e\in E}\in L^2((0,T),L^2(\Gamma))$ such that, up to a subsequence, $\bu^\eps \rightharpoonup \bU$ in $L^2((0,T),L^2(\Gamma))$. In particular $u_e^\eps \rightharpoonup U_e\ \text{in }\ L^2((0,T),L^2(e))
$ for any any edge $e\in E$.

Let us consider an arbitrary edge $e$ parametrized by $[0,l]$ or $[0,\infty)$. In both cases by Lemma \ref{convergence.h1.time} the estimate above guarantees that for any $T>0$
\[
 \eps^{-3} \int_0^T \int_{e} \int_{e} J\Big(\frac{d(x,y)}\eps \Big) (u^\eps(t,y)-u^\eps(t,x))^2
dxdydt\leq
\int_{\Gamma} u_0^2(x)dx
\]
and then $\bU\in L^2((0,T),H^1(e))$ such that
$$
u^\eps \rightarrow U^e\ \text{in }\ L^2((0,T),L^2_{loc}(e)).
$$
This shows that $\bU\in L^2((0,T),H^1(\Gamma))$.
Moreover, in view of the results in Lemma \ref{convergence.h1.time}, for any $\varphi\in L^2((0,T),H^1(e))$ we have
\begin{equation}
\label{limit.1.edge}
\begin{array}{l}
  \displaystyle   \eps^{-3} \int_0^T \int_{e} \int_{e} J\Big(\frac{d(x,y)}\eps \Big) (u^\eps(t,y)-u^\eps(t,x))(\varphi(t,y)-\varphi(t,x))dxdydt \\[10pt]
   \qquad \displaystyle \rightarrow A(J) \int_0^T  \int_{e} U_x(t,x)\varphi_x (t,x) dxdt, \qquad \mbox{ as } \eps \to 0,
   \end{array}
\end{equation}
where $A(J)=\int_{\rr} J(z)z^2dz.$

{\bf Step II.} We now show that $\bU\in L^2((0,T),D(\mathcal{Q}_\Gamma))$, i.e. it belongs to
$L^2((0,T),H^1(e))$ and $\bU(t)$ is also continuous at any vertex for a.e. $t>0$.

Let us now consider two edges  $e$ and $e'$ that have a common vertex $v$. We first prove that the limit function $\bU=(U_e)_{e}$ is continuous at any internal vertex, i.e. $U_e(t,v)=U_{e'}(t,v)$, for a.e. $t$.  Let us assume that the two edges are parametrized by $I\subset (-\infty,0]$ and $I'\subset [0,\infty)$. In view of energy estimate \eqref{energy.relax} it follows that the function $w^\eps$ defined by
\[
	w^\eps(t,x)=
\begin{cases}
u^\eps_e(t,x), & x\in I,\\[5pt]
u^\eps_{e'}(t,x), & x\in I',
\end{cases}
\]
satisfies
\[
\eps^{-3}\int _0^T \int _{I\cup I'} \int _{I\cup I'} J\Big(\frac{x-y}\eps \Big) (w^\eps(t,y)-w^\eps(t,x))^2dxdydt\leq \|u_0\|_{L^2(\Gamma)}^2.
\]
Thus $w^\eps$
converges, up to a subsequence, weakly to a function $w\in L^2((0,T), H^1(I\cup I'))$. From Step I we know that $u^\eps_e$ converges to $u_e$, $u^\eps_{e'}$ converges to $u_{e'}$
where the function obtained from the pair $u_e, u_{e'}$ belongs to $L^2((0,T), H^1(I\cup I'))$. Hence, it follows that $$u_e(t,0-)=u_{e'}(t,0+)$$ for a.e. $t>0$.
Thanks to this property the limit $\bU$ belongs to $L^2((0,T), D(\mathcal{Q}_\Gamma))$ for any $T>0$.
Moreover, by Lemma \ref{adjoint.intervals}, for any $\varphi\in L^2((0,T), D(\mathcal{Q}_\Gamma))$, we have
\[
 \eps^{-3} \int_0^T \int_{e} \int_{e'} J\Big(\frac{d(x,y)}\eps \Big)  (\varphi(t,y)-\varphi(t,x))^2dxdydt\rightarrow 0
\]
and then we get
\begin{equation}
\label{limit.2.edge}
    \eps^{-3} \int_0^T \int_{e} \int_{e'} J\Big(\frac{d(x,y)}\eps \Big) (u^\eps(t,y)-u^\eps(t,x))(\varphi(t,y)-\varphi(t,x))dxdydt\rightarrow 0.
\end{equation}

{\bf Step III.} Let us consider two edges $e$ and $e'$ which do not have a common endpoint. In this case we will prove that for any function $\varphi \in L^2((0,T), L^2(e\cup e'))$
\[
 \eps^{-3}\int _0^T\int_e \int_{e'}J\Big(\frac{d(x,y)}\eps \Big)  (\varphi(t, y)-\varphi(t, x))^2dxdydt\rightarrow 0, \ \text{as}\ \eps\rightarrow 0.
\]
As a consequence using \eqref{energy.relax}  we obtain
\begin{equation}
\label{limit.3.edge}
    \eps^{-3} \int_0^T \int_{e} \int_{e'} J\Big(\frac{d(x,y)}\eps \Big) (u^\eps(t,y)-u^\eps(t,x))(\varphi(t,y)-\varphi(t,x))dxdydt\rightarrow 0,
     \ \text{as}\ \eps\rightarrow 0.
\end{equation}
Let us now prove the first limit. Indeed, we have
\begin{align*}
\label{}
   \eps^{-3}\int_0^T& \int_e \int_{e'}J\Big(\frac{d(x,y)}\eps \Big)  (\varphi(t,y)-\varphi(t,x))^2dxdydt\\
   &
 \leq  2 \eps^{-3}\int_0^T\int_{e'}\varphi^2(t,y) \int_e J\Big(\frac{d(x,y)}\eps \Big) dx dy dt +
 2 \eps^{-3}\int_0^T \int_{e}\varphi^2(t,x) \int_{e'} J\Big(\frac{d(x,y)}\eps \Big) dy dx dt.
\end{align*}
By symmetry it is sufficient to consider only the first term in the right hand side.
Assume that the distance in the graph between the two edges is $\alpha>0$ and that $e$ is parametrized by $(0,l_e)$. Then  $d(x,y)\geq x+\alpha$ for any $y\in e'$ and since  $J$ is a non-increasing function we have
\[
\eps^{-3} \int _{e}J\Big(\frac{d(x,y)}\eps \Big)dx\leq \eps^{-3} \int _0^{l_e} J\Big(\frac{x+\alpha}\eps \Big)dx=\eps^{-2}\int _{\alpha/\eps}^{l_e/\eps} J(z)dz\leq
\frac{1}{\alpha^2} \int _{\alpha/\eps}^{l_e/\eps} J(z)z^2dz.
\]
Using that  $J$  has a finite second momentum and that $\varphi\in L^2((0,T),L^2(e))$ we obtain that the considered term tends to zero as $\eps\rightarrow 0$ which proves that \eqref{limit.3.edge} holds.

{\bf Step IV.} We prove that $\bU\in C([0,T],L^2(\Gamma))$. In view of Step II it is sufficient to show that
$\bU_t\in L^2((0,T),D(\mathcal{Q}_\Gamma)')$.  We show that $\bU_t^\eps$ is uniformly bounded in $L^2((0,T),D(\mathcal{Q}_\Gamma)')$ so in the limit we obtain the desired property for $\bU$.

Let us take a function $\varphi\in  L^2((0,T),D(\mathcal{Q}_\Gamma))$. Using identity \eqref{energy.relax}, it follows that
\begin{align*}
\Big |\int_0^T \int_{\Gamma} \bU_t^\eps \varphi  \Big|^2&= \Big| \frac{\eps^{-3}}2\int_0^T \int_\Gamma\int_\Gamma
J(\frac{d(x,y)}\eps) (\bU^e(t,y)-\bU^e(t,x))(\varphi(t,y)-\varphi(t,x) dx dy dt \Big|^2\\
&\leq \|\bu_0\|_{L^2(\Gamma)}^2 \frac{\eps^{-3}}4\int_0^T \int_\Gamma\int_\Gamma
J(\frac{d(x,y)}\eps) (\varphi(t,y)-\varphi(t,x))^2dxdydt.
\end{align*}
It is sufficient to show that  the following holds for any $\psi \in D(\mathcal{Q}_\Gamma)$ and $\eps>0$:
\begin{equation}\label{ineg.h1.new}
 \eps^{-3}\int_\Gamma\int_\Gamma
J(\frac{d(x,y)}\eps) (\psi(y)-\psi(x))^2dxdy\leq C(\Gamma,J) \int _{\Gamma} (\psi_x^2+\psi^2)dx.
\end{equation}
To prove that we split the integral in the left hand side in integrals over adjacent edges or not. Let us take two edges $e$ and $e'$ having no common point. In view of Step III
\begin{equation}\label{edges.1}
 \eps^{-3}\int_{e}\int_{e'}
J\Big(\frac{d(x,y)}\eps \Big) (\psi(y)-\psi(x))^2dxdy\leq C(\Gamma,J) \int _{e\cup e'}  \psi^2(x)dx.
\end{equation}
When two edges $e$ and $e'$ have a common endpoint we can parametrize them as in Step II and then $d(x,y)=|x-y|$, $\psi\in H^1(I\cup I')$ and we can use the real line case (see for example \cite[Th.~1]{MR3586796})
\begin{align*}
 \eps^{-3}\int_{e\cup e'}\int_{e\cup e'}
J\Big(\frac{d(x,y)}\eps \Big) (\psi(y)-\psi(x))^2dxdy&= \eps^{-3}\int_{I\cup I'}\int_{I\cup I'}
J\Big(\frac{|x-y|}\eps\Big) (\psi(y)-\psi(x))^2dxdy\\
&
\leq C(I, I',J) \int _{I\cup I'} ( \psi^2+\psi_x^2)dx.
\end{align*}
Hence \eqref{ineg.h1.new} holds and as a consequence
\[
\| \bU_t^\eps\|_{L^2((0,T),D(\mathcal{Q}_{\Gamma})')}\leq C(\Gamma,J)\|\bu_0\|_{L^2(\Gamma)}.
\]

{\bf Step V.} Let us now consider $\bu^\eps\in C([0,\infty),L^2(\Gamma))$ solution of problem \eqref{eq.eps.77}. We multiply equation \eqref{eq.eps.77} by a function $\varphi\in C_c([0,\infty),D(\mathcal{Q}_\Gamma))$, $\varphi_t\in C_c([0,\infty), L^2)$. It follows that,
\begin{align*}
\label{}
 \int_{0}^\infty \int _\Gamma \bu^\eps \varphi_t dxdt
 +\frac {\eps^{-3} }2\int_0^\infty   \int_{\Gamma} \int_{\Gamma} J\Big(\frac{d(x,y)}\eps \Big) (\bu^\eps(t,y)-\bu^\eps(t,x))&(\varphi(t,y)-\varphi(t,x))dxdydt \\
&+ \int_{\Gamma} \bu_0(x)\varphi(0,x)dx =0.
\end{align*}
Since $\bu^\epsilon\rightharpoonup \bU$ in $L^2((0,\infty),L^2(\Gamma))$ we have
\[
 \int_{0}^\infty \int _\Gamma \bu^\eps \varphi_t dxdt \rightarrow \int_{0}^\infty \int _\Gamma \bU \varphi_t dxdt.
\]
In view of Step I and Step III we get under the assumption \eqref{second.momentum.J}
\begin{align*}
 \frac {\eps^{-3} }2\int_0^\infty   \int_{\Gamma} \int_{\Gamma} J\Big(\frac{d(x,y)}\eps \Big) (\bu^\eps(t,y)-\bu^\eps(t,x))&(\varphi(t,y)-\varphi(t,x))dxdydt
\\
 &\rightarrow
 \int_0^\infty \int _{\Gamma} \bU_x (t,x)\varphi_x (t,x) dxdt.
\end{align*}
Hence  the limit point $\bU\in C([0,\infty),L^2(\Gamma))\cap  L^2_{loc}((0,\infty), D(\mathcal{Q}_\Gamma))$ satisfies
\[
  \int_{0}^\infty \int _\Gamma \bU (t,x) \varphi_t (t,x) dxdt
  -\int _0^\infty \int_\Gamma \bU_x(t,x) \varphi_x (t,x) dx dt
  +\int_{\Gamma} \bu_0(x)\varphi(0,x)dx=0.
\]
Classical arguments for the classical heat equation shows that
 $\bU$ is the unique solution to the heat equation in $\Gamma$ with initial datum $\bu_0\in L^2(\Gamma)$.

 The proof is now complete.
\end{proof}

\subsection{Asymptotic behaviour for the nonlocal evolution problem} \label{sect-nonlocal}

Now our goal is to analyze the behaviour of solutions to the nonlocal problem.

Before entering into the statements and proofs of our main results let us prove two auxiliary results
that will be needed in order to obtain decay bounds for the solutions and the asymptotic behaviour.
For the first one we follow ideas from \cite{MR2542582} but adapted to the graphs having some infinite edges.

\subsubsection{Preliminaries}
We now give a decomposition similar to the one done in \cite{MR2542582} but on half line intervals. We assume that function $J$ belongs to $L^1(\rr)$ and it is positive in a neighborhood of the origin. To simplify the presentation we introduce the bilinear form
\[
E^J_I(u,v)=\int_I \int_I  J(x-y)(u(x)-u(y))(v(x)-v(y))dxdy.
\]

\begin{lemma}
	\label{decom.1}Let $I$ be the half line $(0,\infty)$ or the real line $\rr$.
For any $u\in L^2(I)$ there exists a decomposition $u=v+w$ such that
\begin{equation}
\label{est.v}
  \|v_x\|_{L^2(I)}^2+  \|w\|_{L^2(I)}^2\leq C(J) E^J_I(u,u)
\end{equation}
 and for any $a>0$ their norms satisfy
\begin{equation}
\label{est.v.p.loc}
  \|v\|_{L^p(0,a)}+ \|w\|_{L^p(0,a)}\leq C(J) \|u\|_{L^p(0,a+1)}, \qquad \forall\ 1\leq p\leq \infty.
  \end{equation}
\end{lemma}

\begin{proof}The case when  $I$ is  the whole real line  has been proved in  \cite[Th.~2.1]{MR2542582}. We now consider the case $I=(0,\infty)$.

Since $J$ is positive near the origin we can choose a smooth nonnegative function	
	 $\rho$ to be  supported in $(-1,0)\cap {\rm supp} \ J$,  with $\int _{-1}^0\rho =1$ and satisfying
	$$|\rho(z)|+|\rho'(z)|\leq c(J) |J(z)|,\quad  z\in \rr.$$
{	In particular this implies that
\begin{equation}
\label{rho-J}  \|\rho\|_{L^1(\rr)}+\|\rho'\|_{L^1(\rr)}\leq C(J).
\end{equation}
}
	
	For any $x>0$ we set
	\[
	v(x)=\int_0^\infty \rho(x-y)u(y)dy=\int_x^{x+1} \rho(x-y)u(y)dy.
	\]
	Using H\"older inequality integrating with respect to the measure $\rho(x-y)dy$ we immediately obtain that
	\[
	\|v\|^p_{L^p((0,a))}\leq \|\rho\|_{L^1(\rr)} \int _0^a \int _x^{x+1}  |\rho(x-y)||u(y)|^pdydx\leq C(J) \|u\|_{L^p(0,a+1)}^p.
	\]
	{Letting} $w=u-v$ we obtain the last property.
	
	We now prove that this decomposition satisfies the first  property.
Observe that
	since $\rho$ is compactly supported in $(-1,0)$ we have for any $x>0$ that
	$$
	\int _0^\infty \rho'(x-y)dy=\int _{-\infty}^x \rho'(z)dz=\rho(x)=0.
	$$
	Thus
		\[
	v_x(x)=\int_0^\infty \rho'(x-y)u(y)dy=\int _0^\infty \rho'(x-y)(u(y)-u(x))dy
	\]
and
	\[
	|v_x(x)|^2\leq \int _0^\infty |\rho'(x-y)|(u(y)-u(x))^2dy \int _0^\infty |\rho'(x-y)|dy.
	\]
	{ Using \eqref{rho-J} and the fact that $|\rho'(z)|\leq C(J)|J(z)|$ it follows that}
	\[
	\int _0^\infty |v_x(x)|^2dx \leq \|\rho'\|_{L^1(\rr)}\int _0^\infty \int _0^\infty|\rho'(x-y)|(u(y)-u(x))^2dydx\leq C(J) E^J_I(u,u)	.\]	
	On the other hand, since $\int _{-1}^0\rho =1$ and it is supported in $(-\infty,0)$ we obtain
	 $$\int _0^\infty \rho(x-y)dy=\int _{-\infty}^x\rho(z)dz=\int _{-\infty}^0 \rho(z)dz$$
	 and thus function	
	$w=u-v$ can be written as
	\[
	w(x)=
	\int _0^\infty (u(x)-u(y))\rho(x-y)dy, \ x>0,
	\]
{ Using  that $|\rho'(z)|\leq C(J)|J(z)|$ and \eqref{rho-J}  it follows that $w$ satisfies}
	\[
	\int _0^\infty |w(x)|^2dx\leq \int_{-\infty}^0 |\rho(z)|dz  \int _0^\infty \int _0^\infty|\rho(x-y)|(u(y)-u(x))^2dydx\leq C(J) E^J_I(u,u).	\]
	This finishes the proof.
\end{proof}

\begin{lemma}
	\label{nash.like}Let $I$ be the half line $(0,\infty)$ or the real line $\rr$. For any $p\in (1,\infty)$ and $u\in L^p(I) $ it holds
	\begin{equation}
\label{p.nash}
  \|u\|_{L^p(I)}^p \leq C(p,J) \Big(\|u\|_{L^1(I)}^{\frac{2p}{p+1}}
  E^J_I(|u|^{p/2},|u|^{p/2})^{\frac{p-1}{p+1}}  + E_I^J(|u|^{p/2},|u|^{p/2}) \Big).
\end{equation}
	\end{lemma}
	
The case $p\geq 2$ follows from \cite[Th.~1.1]{CM} so the our contribution here is to deal with the case $p\in (1,2)$. Extension to any dimension
considering unbounded exterior domains can be done but it is out of the scope of this article. 	
	
	\begin{proof}
	Let us consider $p\in (1,2)$ since the other cases have been proved in \cite{CM}. Let $u\in L^p(I)$. Then $|u|^{\frac{p}{2}}\in L^2(I)$.
	We consider the decomposition of $|u|^{\frac{p}{2}}=v+w$ as in Lemma \ref{decom.1} or \cite[Th.~2.1]{MR2542582} (when $I=\rr$). In both cases $v=\rho\ast |u|^{\frac{p}{2}}:=\int_I \rho(x-y)|u|^{\frac{p}{2}}(y)dy$ and since $\frac{2}{p}>1$ we have
	\[
	\|v\|_{L^{\frac{2}{p}}(I)}=\|\rho\ast |u|^{\frac{p}2}\|_{L^{\frac{2}{p}}(I)}\leq
	\|\rho\|_{L^1(I)}\| |u|^{\frac{p}2}\|_{L^{\frac{2}{p}}(I)}\leq C(J)\| u\|_{L^1(I)}^{\frac{p}2}.
	\]
 Also by Lemma \ref{decom.1}
  \[
  \|v_x\|_{L^2(I)}^2+\|w\|^2_{L^2(I)} \leq C(J)  E_I^J(|u|^{\frac{p}{2}},|u|^{\frac{p}{2}}).
  \]
 Using the interpolation inequality
and that $\|v\|_{L^\infty(I)}^2\leq 2\|v_x\|_{L^2(I)} \|v\|_{L^2(I)}$    we find
  \begin{align*}
\label{}
  \|v\|_{L^2(I)}&\leq \|v\|_{L^{2/p}(I)}^{\frac 1p}\|v\|_{L^\infty(I)}^{\frac{p-1}p}\leq C(J) \|u\|_{L^1(I)}^{\frac 12} \|v_x\|_{L^2(I)}^{\frac{p-1}{2p}}  \|v\|_{L^2(I)}^{\frac{p-1}{2p}}\\
  &\leq
  C(J) \|u\|_{L^1(I)}^{\frac{1}2} E_I^J(|u|^{\frac{p}2},|u|^{\frac{p}2})^{\frac{p-1}{4p}}  \|v\|_{L^2(I)}^{\frac{p-1}{2p}}.
\end{align*}
 It implies that
 \[
  \|v\|_{L^2(I)}^2\leq C(J) \|u\|_{L^1(I)}^{\frac{2p}{p+1}}
   E_I^J(|u|^{\frac{p}2},|u|^{\frac{p}2})^{\frac{p-1}{p+1}}.
 \]
  Finally we get
  \begin{align*}
\label{}
  \|u\|_{L^p(I)}^p&=\| |u|^{\frac{p}2}\|^2_{L^2(I)} \leq
  2(\|v\|_{L^2(I)}^2 +\|w\|^2_{L^2(I)})\\
  &\leq C(J) \Big(\|u\|_{L^1(I)}^{\frac{2p}{p+1}}
   E(|u|^{\frac{p}2},|u|^{\frac{p}2})^{\frac{p-1}{p+1}} +  E(|u|^{\frac{p}2},|u|^{\frac{p}2})\Big)
\end{align*}
	which finishes the proof.
	\end{proof}

\subsubsection{Decay of the solutions.}
In this section, our main result read as follows: we prove that the problem is well posed and a bound for the decay of solutions.

\begin{theorem}
	\label{th.nonlocal.decay.1}
	For any  $\bu_0\in L^1(\Gamma)\cap  L^p(\Gamma) $, $1\leq p<\infty$, the solution $\bu$ of system \eqref{eq.1} satisfies
\begin{equation}
\label{decay.nonlocal.lp}
  \|\bu(t)\|_{L^p(\Gamma)}\leq \frac {C(\|\bu_0\|_{ L^1(\Gamma)},\|\bu_0\|_{ L^p(\Gamma)})}{(1+t)^{\frac 12(1-\frac 1p)}}, \quad \forall\ t>0.
\end{equation}
In addition, for any  $\bu_0\in L^1(\Gamma)\cap  L^2(\Gamma)$
\begin{equation}
	\label{grad.estimate.2}
 \int _{\Gamma}\int_{\Gamma} J(d(x,y)) (\bu(t,x)-\bu(t,y))^2dxdy\leq
 C(\|\bu_0\|_{ L^1(\Gamma)},\|\bu_0\|_{ L^2(\Gamma)}) t^{-\frac32}.
\end{equation}
\end{theorem}

\begin{proof}[Proof of Theorem \ref{th.nonlocal.decay.1}]
Observe that property \eqref{grad.estimate.2} is a consequence of \eqref{grad.estimate} and  of the decay property
\eqref{decay.nonlocal.lp}. Indeed, by \eqref{grad.estimate} the left hand side of \eqref{grad.estimate.2} satisfies
\[
 \mathcal{E}^J_\Gamma(\bu(2t),\bu(2t))\leq \frac{\|\bu(t)\|^2_{L^2(\Gamma)}}{t}\leq
 C(\|\bu_0\|_{ L^1(\Gamma)},\|\bu_0\|_{ L^2(\Gamma)}) t^{-\frac32}.
\]

We now prove the decay property \eqref{decay.nonlocal.lp}.
We use the results obtained in the previous section and the following energy estimate obtained by multiplying \eqref{eq.1} with $|\bu|^{p-2}\bu$ and integrating in the space variable
  \begin{align}
\label{energy.nonlocal.pp}
  \frac 1p \frac {d}{dt}&\int _{\Gamma} |\bu|^p(t,x)dx\\
  &=-\frac 12 \int _{\Gamma}\int_{\Gamma}J(x-y)  (\bu(t,x)-\bu(t,y))( |\bu(t,x)|^{p-2}\bu(t,x)-|\bu(t,y)|^{p-2}\bu(t,y))dxdy\\
  &\leq -c(p) \int _{\Gamma}\int_{\Gamma} J(x-y)( |\bu(t,x)|^{\frac{p}2}-|\bu(t,y)|^{\frac{p}2})^2dxdy
  =-c(p)\mathcal{E}^J_\Gamma ( |\bu(t)|^{\frac{p}2},|\bu(t)|^{\frac{p}2}).
\end{align}

Let us now fix an edge $e$. For each such edge we choose a path $\Gamma_e$ that connects $e$ with $\Gamma_\infty$. It may happen to exist many such paths but we choose one of them. When an edge $e$ has infinite length we can chose $\Gamma_e$ to be exactly $e$.
	This path $\Gamma_e$ can be parametrized by the infinite interval $I_e=[0,\infty)$.
	We set $u_e$ to be the restriction of $\bu$ to $\Gamma_e$. It is clear that
	\[
	\|u_e(t)\|_{L^1(\Gamma_e)}\leq \|\bu (t)\|_{L^1(\Gamma)}\leq \|\bu_0\|_{L^1(\Gamma)}.
	\]
	We apply Lemma \ref{nash.like} to each function $u_e(t)$ and obtain
\begin{align*}
\label{}
  \|u_e\|_{L^p(\Gamma_e)}^p &\leq C(p,J) \Big(\|u_e\|_{L^1(\Gamma_e)}^{\frac{2p}{p+1}}
  \mathcal{E}^J_{\Gamma_e} (|u_e|^{p/2},|u_e|^{\frac{p}2})^{\frac{p-1}{p+1}}  +  \mathcal{E}^J_{\Gamma_e} (|u_e|^{\frac{p}2},|u_e|^{\frac{p}2}) \Big)\\
  &\leq C(p,J) \Big(\|\bu_0\|_{L^1(\Gamma)}^{\frac{2p}{p+1}}
  \mathcal{E}^J_\Gamma (|\bu|^{\frac{p}2},|\bu|^{\frac{p}2})^{\frac{p-1}{p+1}}  +  \mathcal{E}^J_\Gamma (|\bu|^{\frac{p}2},|\bu|^{\frac{p}2}) \Big).
\end{align*}
	Using that $u^e=(u_e)_{|_e}=\bu_{|_e}$ and summing over all the edges (a finite number) of graph $\Gamma$ we get
	\begin{align*}
\label{}
  \|\bu(t)\|_{L^p(\Gamma)}^p&=\sum _{e\in E}  \|u^e\|^p_{L^p(e)}\leq
 \sum _{e\in E}   \|u_e\|_{L^p(\Gamma_e)}^p\\
 &\leq |E|C(p,J) \Big(\|\bu_0\|_{L^1(\Gamma)}^{\frac{2p}{p+1}}
  \mathcal{E}^J_\Gamma (|\bu(t)|^{\frac{p}2},|\bu(t)|^{\frac{p}2})^{\frac{p-1}{p+1}}  +  \mathcal{E}^J_\Gamma( |\bu|^{\frac{p}2},|\bu|^{\frac{p}2}) \Big)\\
  &=f(\mathcal{E}^J_\Gamma( |\bu|^{\frac{p}2},|\bu|^{\frac{p}2}) ),
\end{align*}
where
\[
f(s)= |E|C(p,J)  (\|\bu_0\|_{L^1(\Gamma)}^{\frac{2p}{p+1}}
  s^{\frac{p-1}{p+1}}  +  s \Big).
\]
This shows that
\[
  \frac 1{pc(p)} \frac {d}{dt}(\|\bu(t)\|_{L^p(\Gamma)}^p)\leq -f^{-1}(\|\bu(t)\|_{L^p(\Gamma)}^p).
\]
Using that $f(t)\simeq t^{\frac{p-1}{p+1}}$ as $t\simeq 0$ we obtain that
$f^{-1}(t)\simeq t^{\frac{p+1}{p-1}}$ as $t\simeq 0$.
The same arguments as in \cite[Lemma~3.1]{MR2542582} give us the desired decay estimate.
\end{proof}

\subsubsection{Asymptotic behaviour.}
Now, let us rescale the solution as we did for the local case, let
$\bu_\lambda:
\Gamma^\lambda\rightarrow\rr$, $\lambda>0$ be given by
\[
\bu_\lambda(t,x)=\lambda \bu(\lambda^2t, \lambda x), \quad x\in \Gamma^\lambda, t>0.
\]
It follows that $\bu_\lambda$ satisfies
	\begin{equation}
\label{eq.non.lambda}
  \left\{
  \begin{array}{ll}
 \displaystyle  \partial_t \bu_\lambda(t,x)=\lambda^3 \int_{\Gamma^\lambda} J(\lambda d(x,y)) (\bu_\lambda(t,y)-\bu_\lambda(t,x) )dy,\
 x\in \Gamma^\lambda, t>0,\\[10pt]
    \bu_\lambda(x,0)=\lambda \bu_0 (\lambda x), \ x\in \Gamma^\lambda.
  \end{array}
  \right.
\end{equation}
Using the estimates obtained in Theorem \ref{th.nonlocal.decay.1} for $\bu$ we get that $\bu_\lambda$ satisfies the following uniform estimates:

\begin{lemma}
	\label{uniform.nonocal.est}Let $1\leq p<\infty$. For any $\bu_0\in L^1(\Gamma)\cap L^p(\Gamma)$ it holds that
	\begin{equation}
\label{norma.p.non.la}
  \|\bu_\lambda(t)\|_{L^p(\Gamma^\lambda)}\leq C(p, \|\bu_0\|_{L^1(\Gamma)}, \|\bu_0\|_{L^p(\Gamma)}) t^{-\frac 12(1-\frac1p)},\  \forall t>0.
\end{equation}
Moreover, for any $\bu_0\in L^1(\Gamma)\cap L^2(\Gamma)$ the following bound
	\begin{equation}
\label{nest.energy.lambda}
\lambda^3   \int_{\Gamma^\lambda} \int_{\Gamma^\lambda}
J(\lambda d(x,y))(\bu_\lambda(t,x)-\bu_\lambda(t,y))^2dxdy \leq C(p, \|\bu_0\|_{L^1(\Gamma)}, \|\bu_0\|_{L^2(\Gamma)} ) t^{-\frac32},\quad \forall \ t>0,
\end{equation}
holds uniformly in $\lambda>0$.
\end{lemma}

Before proving Theorem \ref{th.nonlocal.1.intro} we need some auxiliary results.
For a function $\varphi\in H^1(\Gamma_\infty)$ such that
$\varphi^e(0)=\varphi^{e'}(0)$ for all $e,e'\in \Gamma_\infty$, denote this common value by $\varphi(0)$.
We extend it to function defined on the whole graph $\Gamma^\lambda$,
$\tilde \varphi_\lambda:  \Gamma^\lambda\rightarrow\rr$, such  that to be constant on the finite part of the graph,
i.e.
\[
\tilde \varphi_\lambda=
\begin{cases}
	\varphi, & \text{on}\ \Gamma_\infty,\\
	\varphi(0), & \text{on}\ \Gamma_f^\lambda.
\end{cases}
\]

\begin{lemma}
	\label{termino.cruzado}
	There exists a non-increasing function $H\in L^\infty([0,\infty))$ going to zero at infinity such that
	for any $\varphi\in H^1(\Gamma_\infty))$  such that
$\varphi^e(0)=\varphi^{e'}(0)$ for all $e,e'\in \Gamma_\infty$ we have
\begin{equation}
\label{est.extension.1}
  \lambda^3 \int_{\Gamma_\infty} \int_{\Gamma_f^\lambda} J(\lambda d(x,y))(\tilde \varphi(x)-\tilde\varphi(y) )^2dydx\leq \int_{\Gamma_\infty} \varphi^2_x(x)H(\lambda |x|)dx, \quad \forall\lambda>0.
\end{equation}
\end{lemma}

\begin{proof}
	By the way we define the extension function $\tilde \varphi$ we have that it is constant on $\Gamma_f^\lambda$. Then
	\[
	 \lambda^3 \int_{\Gamma_\infty} \int_{\Gamma_f^\lambda} J(\lambda d(x,y))(\tilde \varphi(x)-\tilde\varphi(y) )^2dydx = \lambda^3 \int_{\Gamma_\infty}(\varphi(x)-\varphi(0))^2\Big(\int_{\Gamma_f^\lambda}J(\lambda d(x,y))dy \Big)dx.
	\]
	Take an edge $e$ of $\Gamma_\infty$ parametrized by $[0,\infty)$. For any $x\in e$ and $y\in {\Gamma_f^\lambda} $ we have $d(x,y)\geq |x|$. Since $J$ is a non-increasing function we obtain
	\[
	\int_{\Gamma_f^\lambda}J(\lambda d(x,y))dy\leq |{\Gamma_f^\lambda}| J(\lambda |x|)=C(\Gamma_f)\lambda^{-1}J(\lambda |x|).
	\]
	It is then sufficient to consider integrals of the type
	\[
	I=\lambda^2 \int_0^\infty (\varphi(x)-\varphi(0))^2 J(\lambda |x|) dx.
	\]
	Using that $(\varphi(x)-\varphi(0))^2\leq x \int_0^x\varphi_x^2(s)ds$ we obtain that
	\[
	I\leq   \int_0^\infty\varphi_x^2(x) \int_{\lambda x}^\infty zJ(z)dz.
	\]
	Denoting $H(s)=\int_s^\infty zJ(z)dz$ and using that $J\in L^1(\rr,1+|x|^2)$ we obtain the desired result.
\end{proof}

\begin{lemma}
	For any $\varphi\in H^1(\Gamma_\infty)$  such that
$\varphi^e(0)=\varphi^{e'}(0)$ for all $e,e'\in \Gamma_\infty$ the following hold:\\
i)\[
\mathcal{E}^{J,\lambda}_{\Gamma_\infty}(\varphi,\varphi):=
 \lambda^3 \int_{\Gamma_\infty}\int_{\Gamma_\infty}
J(\lambda d(x,y))( \varphi(x)- \varphi(y) )^2dydx\leq C(J)\|\varphi_x\|^2_{L^2(\Gamma_\infty)},
\] uniformly in $\lambda>0$.

ii)\[
\mathcal{E}^{J,\lambda}_{\Gamma^\lambda}(\tilde\varphi,\tilde\varphi):=  \lambda^3 \int_{\Gamma^\lambda}\int_{\Gamma^\lambda} J(\lambda d(x,y))( \tilde\varphi(x)- \tilde\varphi(y) )^2dydx= \mathcal{E}^{J,\lambda}_{\Gamma_\infty}(\varphi,\varphi)+o(1), \quad \lambda\rightarrow\infty.
\]
\end{lemma}

\begin{proof}

	Using the fact that the extension function $\tilde\varphi$ is constant on $\Gamma_f^\lambda$ we obtain
	\[
	\mathcal{E}^{J,\lambda}_{\Gamma^\lambda}(\tilde\varphi,\tilde\varphi)=\mathcal{E}^{J,\lambda}_{\Gamma_\infty}(\varphi,\varphi)+ 2\lambda^3 \int_{\Gamma_\infty} \int_{\Gamma_f^\lambda} J(\lambda d(x,y))(\tilde \varphi(x)-\tilde\varphi(y) )^2dydx
	\]
	Using estimate \eqref{est.extension.1} and the fact that $\varphi\in H^1(\Gamma_\infty)$ by Lebesgue convergence theorem we obtain that the last term is $o(1)$.

Let us now prove the first part.
When $\Gamma_\infty$ consists only of one edge parametrized by $(0,\infty)$ we extend it to the whole line and use the previous results on the real line \cite[Lemma~2.3]{IgnatIgnat}.

	Let us consider two edges $e$ and $e'$ of $\Gamma_\infty$. It is sufficient to estimate each $I_{e,e'}$ defined by
	\[
	I_{e,e'}=\lambda^3 \int_{e\cup e'}\int_{e\cup e'} J( \lambda d(x,y))( \varphi(x)- \varphi(y) )^2dydx.
	\]
	Assume that the two edges are parametrized by $(-\infty,0]$ and $[0,\infty)$. Then $d(x,y)\geq |y-x|$ and using that $J$ is non-increasing we get (the continuity at $x=0$ guarantees that the function $\varphi$ belongs to $H^1(\rr)$ and we can apply then the result on the real line)
	\[
	I_{e,e'}\leq \lambda^3 \int_{\rr}\int_{\rr} J( \lambda |x-y|)( \varphi(x)- \varphi(y) )^2dydx\leq \int_{\rr}J(z)z^2dz \int_{\rr}\varphi_x^2dx.
	\]
Summing
	this inequality over all the edges $e$ and $e'$ of $\Gamma_\infty$ we obtain the desired estimate.
\end{proof}

Now we are ready to proceed with the proof of the asymptotic behavior of the solutions.
Notice that the obtained bound for the decay in $L^p$ obtained in Theorem \ref{th.nonlocal.decay.1} is the same that holds for the heat equation in $\Gamma$.
Therefore, our next task is to show that we also have the same asymptotic profile.

\begin{proof}[Proof of Theorem \ref{th.nonlocal.1.intro}]
We consider the case $p\geq 2$ in order to avoid technical difficulties in order to obtain compactness.
In particular $\bu_0\in L^1(\Gamma)\cap L^2(\Gamma)$ and we can obtain the compactness of the rescaled family $\bu_\lambda$ using the energy estimate \eqref{nest.energy.lambda}
 When
$1<p<2$ we have to use the compactness arguments in \cite{IgnatIgnat} and to obtain new versions of the Lemmas in Section \ref{sect-Appendix}.

\textbf{Step A. Existence of a limit profile $\bU\in L^2_{loc}((0,\infty), D(\mathcal{Q}_{\Gamma_\infty}))$.}
Let us fix two times, $\tau$, $T$, such that $0<\tau<T<\infty$.
Using the results in Lemma \ref{uniform.nonocal.est}  and Lemma \ref{convergence.h1.time} we obtain that  $\bU=(U^e)_{e\in \Gamma_\infty}\in L^2((\tau, T), H^1(\Gamma_\infty))$
and $\bu_\lambda\rightarrow \bU$ in $L^2((\tau, T), L^2_{loc}(\Gamma_\infty))$.
Indeed using \eqref{nest.energy.lambda} for any edge $e$ of $\Gamma_\infty$
we obtain by Lemma \ref{convergence.h1.time} that  $u^e\in L^2((\tau,T),H^1(e))$,  $u^e_\lambda(t)\rightharpoonup U^e(t)$ in $L^2(e)$ and  $u^e_\lambda(t)\rightarrow U^e(t)$ in $L^2((\tau, T), L^2_{loc}(e))$.

We  prove that in fact
the values of $\bU$ at the endpoints of $e_j$ are equal, that is,
$$U^e(t,j(I(e),e))=U^{e'}(t,j(I(e'),e')),$$
for a.e. $t$ and then $\bU\in L^2((\tau, T), D(\mathcal{Q}_{\Gamma_\infty}))$.
To do that let us consider two edges, $e$ and $e'$, of $\Gamma_\infty$ and $\Upsilon^\lambda$ a path in $\Gamma^\lambda$ which contains the two edges. Since the graph $\Gamma_f$ is finite there exits $\alpha>0$ such that the length of the
path $\Upsilon_f^\lambda=\Gamma^\lambda\cap \Upsilon^\lambda$ is $2\alpha/\lambda$. Let us parametrize $e$ and $e'$ with $(-\infty,0)$ respectively $(0,\infty)$, $ \Upsilon_f^\lambda$ with $(-\alpha/\lambda, \alpha/\lambda)$. In the following we will not meke precise the time dependence of $\bu_\lambda$ unless it is necesarely.
We introduce the family $(f_\lambda)_{\lambda>0}$ defined by
\[
f_\lambda(x)=
\begin{cases}
u_\lambda^e(x+\alpha/\lambda),& x<-\alpha/\lambda,\\
u_\lambda|_{ \Upsilon_f^\lambda}	(x), & |x|<\alpha/\lambda,\\
u_\lambda^{e'}(x-\alpha/\lambda),& x>\alpha/\lambda.
\end{cases}
\]
In view of the properties of $\bu_\lambda(t)$ we obtain that $(f_\lambda)_{\lambda>0}$  is uniformly bounded in $L^2(\rr)$ and it safisfies
\[
\lambda^3 \int_{\rr}\int_{\rr} J(\lambda(x-y))(f_\lambda(x)-f_\lambda(y))^2dxdy\leq C(t).
\]
It follows that there exists $f\in H^1(\rr)$ such that $f_\lambda\rightarrow f$ in $L^2_{loc}(\rr)$ and $f_\lambda \rightharpoonup   f $ in $L^2(\rr)$. In particular $f\in C(\rr)$ and satisfies $|f(b)-f(a)|\lesssim |b-a|^{1/2}$ for all $a<0<b$. We claim that
\[
f(x)=
\begin{cases}
	u^e(x), &x<0,\\
	u^{e'}(x), & x>0.
\end{cases}
\]
This implies that for any $a<0<b$, $|u^e(a)-u^{e'}(b)|\lesssim |b-a|^{1/2}$ and then $u^e(0-)=u^{e'}(0+)$ which proves the continuity of the limit profile $\bU$. It remains to prove the above claim. Let us consider $a<0$, $\varphi$ a smooth function supported in $(-\infty, 0)$ and $\lambda>0$ such that $a<-\alpha/\lambda<0$. Then
\[
\int_{-\infty}^af_\lambda\varphi=\int_{-\infty}^a u^e_\lambda(x+\frac \alpha\lambda)\varphi(x)dx=\int_{-\infty}^0
u^e_\lambda(y) \varphi(y-\frac \alpha\lambda)\chi_{(-\infty,a+\frac \alpha\lambda)}.dy
\]
Using that $f_\lambda\rightharpoonup f$ and $u_\lambda^e\rightharpoonup u^e$ in $L^2((-\infty, 0))$ we let $\lambda\rightarrow \infty$ to obtain that $f=u^e$ on $(-\infty,a)$. Since $a$ is arbitrary we obtain that $f\equiv u_e$ on $(-\infty,0)$. The same argument shows that $f\equiv u_{e'}$ on $(0,\infty)$ and the claim is proved.

\textbf{Step B. Equation satisfied by the limit profile.}
Let us fix $T>0$.
Let us now consider a function $\varphi\in C([0,\infty),H^1(\Gamma_\infty))$ with $\varphi_t\in C([0,\infty),L^2(\Gamma_\infty))$ compactly supported in time in the interval $[0,T]$ and
such that
$\varphi^e(t,0)=\varphi^{e'}(t,0)$ for all $e,e'\in \Gamma_\infty$, i.e. $\varphi\in C([0,\infty),D(\mathcal{Q}_{\Gamma_\infty}))$. Denote this common value by $\varphi(t,0)$.
We extend it to function defined on the whole graph $\Gamma^\lambda$,
$\tilde \varphi_\lambda:[0,\infty)\times \Gamma^\lambda$, function that is constant on the finite part of the graph,
i.e.
\[
\tilde \varphi_\lambda=
\begin{cases}
	\varphi, & \text{on}\ \Gamma_\infty,\\
	\varphi(t,0), & \text{on}\ \Gamma_f^\lambda.
\end{cases}
\]

Multiplying the equation satisfied by $\bu_\lambda$ with $\varphi$ and integrating in time and space we obtain
\begin{align*}
  0= & \int_0^T\int_{\Gamma^\lambda} \bu_\lambda  \varphi_t(t,x)dxdt+ \int_{\Gamma^\lambda} \bu_\lambda(0,x)
  \varphi_\lambda (0,x)dxdt\\
  &\quad +\lambda^3 \int_0^T \int_{\Gamma^\lambda} \int_{\Gamma^\lambda}
  J(\lambda d(x,y))(\bu_\lambda(t,y)-\bu_\lambda(t,x)) \varphi(t,x) dydxdt
  \\
  &=I_1^\lambda+I_2^\lambda -\frac 12 \int_0^T  \mathcal{E}_{\Gamma^\lambda}^{J,\lambda}(\bu_\lambda(t), \varphi(t))dt:=I_1^\lambda+I_2^\lambda+I_3^\lambda
    \end{align*}
  where
  \[
  \mathcal{E}^{J,\lambda}_{\Gamma^\lambda}(\bu_\lambda(t), \varphi)
  ={\lambda^3} \int_{\Gamma^\lambda} \int_{\Gamma^\lambda}
  J(\lambda d(x,y))(\bu_\lambda(t,y)-\bu_\lambda(t,x))(\varphi(t,y)- \varphi(t,x))dxdy.
  \]
We claim that the following hold where
$\bU\in L^2_{lov}((0,\infty),D(\mathcal{Q}_{\Gamma_\infty}))$  is the a limit of $\bu_\lambda$ above
\begin{equation}
\label{conv.i.1}
  I_1^\lambda\rightarrow \int_0^T\int_{\Gamma_\infty} \bU \varphi_t,
\end{equation}
\begin{equation}
\label{conv.i.2}
  I_2^\lambda \rightarrow \varphi(0,0)  M,
\end{equation}
\begin{equation}
\label{conv.i.3}
  I_3^\lambda\rightarrow -\frac 12\int_{\rr}J(z)z^2dz \int_0^T \int_{\Gamma_\infty} \bU_x \varphi_{x} = -\int_0^T \int_{\Gamma_\infty} \bU_x \varphi_{x}.
  \end{equation}
These convergences show that the limit function $\bU\in L^2_{loc}((0,\infty),D(\mathcal{Q}_{\Gamma_\infty}))$ is a solution to
$$
0= \int_0^\infty \int_{\Gamma_\infty} \bU \varphi_t + \varphi(0,0)  M - \int_0^\infty \int_{\Gamma_\infty} \bU_x \varphi_{x}.
$$

Since $\bU(t)\in D(\mathcal{Q}_{\Gamma_\infty})$ for a.e. $t>0$ we have for any $\varphi\in D(\Delta_{\Gamma_\infty})$ that
$$
(\bU_x,\varphi_x)_{L^2(\Gamma_\infty)}=(\bU,\varphi_{xx})_{L^2(\Gamma_\infty)}.
$$
Thus
for  $\varphi\in C([0,\infty), D(\Delta_{\Gamma_\infty}))\cap C^1([0,\infty),L^2(\Gamma_\infty))$ we obtain that the limit point $\bU$ satisfies
\[
0= \int_0^\infty \int_{\Gamma_\infty} \bU (t,x) ( \varphi_t (t,x) + \varphi_{xx} (t,x) )dx dt+ \varphi(0,0)  M,
\]
hence it is a solution to the heat equation in $\Gamma$ with initial condition $M\delta_{x=0}$ and therefore the asymptotic
profile claimed in Theorem \ref{th.nonlocal.1.intro}
follows from our results for the local case.

Therefore, we  have to show \eqref{conv.i.1}, \eqref{conv.i.2} and \eqref{conv.i.3}.
As in the local case we have that $\bU$ is uniformly bounded in $L^2((0,T),L^2(\Gamma_\infty))$ so
\[
\int_0^\infty\int_{\Gamma_\infty} \bu_\lambda \widetilde \varphi_t \rightarrow \int_0^\infty\int_{\Gamma_\infty} \bU  \varphi_t.
\]
Moreover, since $\varphi_t(t,0)$ has compact support in time
\[
 \int_0^\infty\int_{\Gamma^\lambda_f} |\bu_\lambda \widetilde \varphi_t| dxdt\leq
 \int_{0}^\infty \|u_\lambda(t)\|_{L^2(\Gamma_f^\lambda) }|\varphi_t(t,0) | |\Gamma_f^\lambda|^{1/2}\lesssim \lambda^{-\frac{1}{2}} \int_0^\infty t^{-1/4} |\varphi_t(t,0)|dt \rightarrow 0.
\]
This shows \eqref{conv.i.1}.

For the second estimate we use the mass conservation an the fact that $\tilde \varphi_\lambda$ is constant on $\Gamma_f^\lambda$:
\begin{align*}
\label{}
  I^2_\lambda-M\varphi(0,0)&=\int_{\Gamma^\lambda} u^\lambda(0,x)(\tilde \varphi_\lambda(0,x)-\varphi(0,0))dx=\int_{\Gamma_\infty} u^\lambda(0,x)( \varphi (0,x)-\varphi(0,0))dx\\
  &=\int_{\Gamma_\infty} u(0,x)( \varphi (0,x/\lambda)-\varphi(0,0))dx\rightarrow 0.
\end{align*}

%
%

Let us analyze the last term $I_3^\lambda$. We will prove the desired limit in few steps.

\textit{Step 1.} We prove that
\[
\int_0^T \mathcal{E}^{J,\lambda}_{\Gamma^\lambda}(\bu_\lambda(t),\tilde\varphi(t))dt=
\int_0^T \mathcal{E}^{J,\lambda}_{\Gamma_\infty}(\bu_\lambda(t), \varphi(t))dt+o(1).
\]
Indeed, since $\tilde \varphi$ is constant in $\Gamma_f^\lambda$ we have
\begin{align*}
\label{}
|  \mathcal{E}^{J,\lambda}_{\Gamma^\lambda}(\bu_\lambda(t),\tilde\varphi(t))&-\mathcal{E}^{J,\lambda}_{\Gamma_\infty}(\bu_\lambda(t), \varphi(t))|\\
  &\leq 2\lambda^3\int_{\Gamma_\infty} \int_{\Gamma_f^\lambda} J(\lambda d(x,y))|\bu_\lambda(t,x)-\bu_\lambda(t,y)||\tilde \varphi(t,x)-\tilde\varphi(t,y) |dydx\\
  &\leq 2  (\mathcal{E}^{J,\lambda}_{\Gamma^\lambda}(\bu_\lambda(t),(\bu_\lambda(t)))^{1/2}\Big( \lambda^3 \int_{\Gamma_\infty} \int_{\Gamma_f^\lambda} J(\lambda d(x,y))(\tilde \varphi(x)-\tilde\varphi(y) )^2dydx\Big)^{\frac{1}2} \\
  &\lesssim  t^{-\frac{3}{4}} \Big(\int _{\Gamma_\infty} \varphi_x^2(t,x)H(\lambda x)dx  \Big)^{\frac{1}2}.
\end{align*}
Integrating in time the above inequality we obtain that
\[
\int _0^T |  \mathcal{E}^{J,\lambda}_{\Gamma^\lambda}(\bu_\lambda(t),\tilde\varphi(t)) -\mathcal{E}^{J,\lambda}_{\Gamma_\infty}(\bu_\lambda(t), \varphi(t))|dt\lesssim
T^{\frac{1}4}\sup_{t\in [0,T]}  \Big(\int _{\Gamma_\infty} \varphi_x^2(t,x)H(\lambda x)dx  \Big)^{\frac{1}2}.
\]
Using that $ \varphi\in C([0,\infty),H^1(\Gamma_\infty))$ we can apply the dominated convergen theorem to  obtain that the last goes to zero as $\lambda\rightarrow\infty.$

\textit{Step 2.} For any $0<\tau<T$ the integrals over $(0,\tau)$ are small, uniformly in $\lambda$: it holds that
\[
\int_0^\tau |\mathcal{E}^{J,\lambda}_{\Gamma_\infty}(\bu_\lambda(t), \varphi(t))|dt \lesssim \tau^{\frac{1}4} \|\varphi\|_{L^\infty([0,T],H^1(\Gamma_\infty))},
\]
and
\[
\Big|\int_0^\tau \int_{\Gamma_\infty} \bU_x(t,x)\varphi_x(t,x)dxdt\Big| \leq \tau^{\frac{1}4} \|\varphi\|_{L^\infty([0,T],H^1(\Gamma_\infty))}.
\]

To check the first one notice that, in view of \eqref{nest.energy.lambda}, we have
\[
|\mathcal{E}^{J,\lambda}_{\Gamma_\infty}(\bu_\lambda(t), \varphi(t))| \leq \mathcal{E}^{J,\lambda}_{\Gamma_\infty}(\bu_\lambda(t), \bu_\lambda(t))^{\frac{1}2} \mathcal{E}^{J,\lambda}_{\Gamma_\infty}(\varphi(t), \varphi(t))^{\frac{1}2}
\lesssim t^{-\frac{3}4} \|\varphi_x(t)\|_{L^2(\Gamma_\infty)}.
\]
Integrating in $[0,\tau]$ we obtain the desired estimate.

For the second limit
remark that for each $t>0$, we have
\[
\mathcal{E}^\lambda_{\Gamma_\infty}(\bu_\lambda(t), \varphi(t))\rightarrow \int_{\Gamma_\infty} \bU_x(t,x)\varphi_x(t,x)dx.
\]
Also $|\mathcal{E}^\lambda_{\Gamma_\infty}(\bu_\lambda(t), \varphi(t))|\leq t^{-\frac{3}4} \|\varphi\|_{L^\infty([0,T],H^1(\Gamma_\infty))}\in L^1((0,T))$. The dominated convergence theorem applied on the time interval $(0,\tau)$  gives us the desired result.

\textit{Step 3.} Let us choose $0<\tau<T$. On the interval $[\tau,T]$ we apply the third part of Lemma \ref{convergence.h1.time} to $D=(0,\infty)$ to obtain
 that
\[
 \int_\tau^T \mathcal{E}^\lambda_{\Gamma_\infty}(\bu_\lambda(t), \varphi(t))dt \rightarrow
\int_{\rr} J(z)z^2dz\int_\tau^T \int_{\Gamma_\infty} U_x(t,x)\varphi_x(t,x)dxdt.
\]
Thus in view of Step II we obtain  \eqref{conv.i.3}.

\textbf{Step C. Tail control and conclusion}. Using the arguments for nonlocal problems in \cite[Lemma 2.7]{IgnatIgnat} together with the ones in Lemma \ref{estimates.ulambda} to control the tail \eqref{tail.control}
we obtain similar results for the solutions $\bu_\lambda$ of the nonlocal problem.
It means that the local convergence obtained at Step A is not only local but it holds in $L^1(\Gamma_\infty)$:
for some $t_0>0$ it holds
\[
\bu_\lambda(t_0)\rightarrow \bU_M(t_0) \, \text{in}\ L^1(\Gamma_\infty).
\]
Then \eqref{limit.p.intro} holds for $q=1$. The other cases follows by using the strong convergence in $L^1(\Gamma_\infty)$ together with the decay of the solutions in $L^p(\Gamma)$. Indeed, choosing $\alpha$ such that
\[
\frac 1q=\frac{\alpha}p+\frac{1-\alpha}{1}.
\]
we obtain
\begin{align*}
 \|\bu(t)-\bU_M(t)\|_{L^q(\Gamma_\infty)}& \leq  \|\bu(t)-\bU_M(t)\|_{L^1(\Gamma_\infty)}^{1-\alpha}
 \|\bu(t)-\bU_M(t)\|_{L^p(\Gamma_\infty)}^{ \alpha}\\
 &\leq  o(1)t^{-\frac 12(1-\frac 1q)}.
\end{align*}

On the compact part of the graph for any $1\leq q<p$
 we trivially have
\begin{align*}
\|\bu(t)-\bU_M(t)\|_{L^q(\Gamma_f)}&\leq C(p,q,\Gamma_f)\|\bu(t)-\bU_M(t)\|_{L^p(\Gamma_f)}\\
&\leq
C(p,q,\Gamma_f)(\|\bu(t)\|_{L^p(\Gamma)}+\|\bU_M(t)\|_{L^p(\Gamma_f)})\\
&\leq C(p,q,\Gamma_f)(t^{-\frac 12(1-\frac 1p)}+t^{-\frac{1}2})=o(t^{-\frac 12(1-\frac 1q)}), \ t\rightarrow \infty.
\end{align*}
The proof is now complete.
\end{proof}

\section{Appendix} \label{sect-Appendix}

In this Appendix we collect some compactness results that were used when studying the relaxation limit and the asymptotic behaviour
for the nonlocal problem. We will use these results in 1-dimension (take $d=1$ below) but we state them in any dimension since the results
hold with greater generality. For the proof we use ideas from \cite{ellibro}. More general assumptions on the function $\rho$
can be found in \cite{ponce}.

\begin{lemma}
	\label{convergence.h1}Let $D\subset \rr^d$ be an open set that has the extension property in $H^1$, for example,
	$D$ is a bounded $C^1$-domain or $\rr^d_+$. Let $\rho:\rr^d\rightarrow\rr$ be a nonnegative $L^1$ radial function having a second momentum in $L^1(\rr^d)$, $\rho>0$ in a neighbourhood  of $x=0$,  and take $$\rho_n(x)=n^d\rho(nx).$$ Let $f_n$ be a sequence in $L^2(D)$ such that
	\begin{equation}
\label{est.1.6}
  n^2 \int_D \int_D \rho_n(x-y)(f_n(x)-f_n(y))^2dxdy\leq M.
\end{equation}
 1. If $f_n\rightharpoonup f $ in $L^2(D)$ then $f\in H^1(D)$
  and
  \[
  F_n(x,z)= (\rho(z))^{\frac{1}2}\chi_D(x+\frac zn) \frac {f_n(x+\frac zn) -f_n(x)}{\frac 1n}\rightharpoonup (\rho(z))^{\frac{1}2} z\cdot \nabla f(x)
  \]
  weakly in $L_x^2(D)\times L^2_z(\rr^d)$.\\
2. For any $f\in H^1(D)$
\begin{equation}
\label{conv.liminf}
  \int_{\rr^d}\rho(z)|z|^2dz \int _{D}|\nabla f(x)|^2dx\leq \liminf_{n\to \infty}  n^2 \int_D \int_D \rho_n(x-y)(f(x)-f(y))^2dxdy
\end{equation}
3. If $\varphi\in H^1(\rr^d)$ then
  \begin{align}
\label{conv.h1}
n^2   \int_D \int_D \rho_n(x-y)(f_n(x)-f_n(y))&(\varphi(x)-\varphi(y))dxdy\\
&\rightarrow
\nonumber \int_{\rr^d}\rho(z)|z|^2dz \int _{D}\nabla f(x)\nabla \varphi(x)dx.
\end{align}
{Notice that by 1. we have $f\in H^1(D)$.}

4. If $D$ is a smooth bounded domain of $\rr^d$ and $\rho(x)\geq \rho(y)$ if $|x|\leq |y|$ then $\{f_n\}_{n}$ is relatively compact in $L^2(D)$.
\end{lemma}

\begin{proof}
	The first and fourth part are in \cite[Th.~6.11, p.~128]{ellibro}.  The second part is exactly equation (36) in \cite{MR1942116}.
	
	For the third part we claim that the following strong convergence holds in $L_x^2(D)\times L^2_z(\rr^d)$:
	\begin{equation}
\label{claim.1}
  \eta_n=    (\rho(z))^{\frac{1}2} \chi_D(x+\frac zn) \frac {\varphi(x+\frac zn) -\varphi(x)}{\frac 1n}\rightarrow 
  (\rho (z))^{\frac{1}2}  z\cdot \nabla \varphi(x).
\end{equation}
Then
\[
(F_n,\eta_n)_{L^2_x(D)\times L^2_z(\rr^d)}\rightarrow \int _{D_x}\int_{\rr^d}\rho(z) (z\cdot \nabla f(x))(z\cdot \nabla \varphi(x))=\int _{ \rr^d} \rho(z)|z|^2dz\int _{D}\nabla f\cdot\nabla \varphi.
\]
Observe that after the change of variables $x+z/n=y$ we get
\begin{align*}
\label{}
  (F_n,\eta_n)_{L^2_x(D)\times L^2_z(\rr^d)}&=n^2
  \int _D \int _{\rr^d} \rho(z) \chi _D(x+\frac zn)(f_n(x+\frac zn) -f_n(x))(\varphi(x+\frac zn) -\varphi(x))dzdx\\
  &=n^{d+2} \int _D \int _{\rr^d} \rho(n(y-x)) \chi _D(y)(f_n(y) -f_n(x))(\varphi(y) -\varphi(x))dydx\\
  &=n^2\int _D \int _{D} \rho_n(y-x) (f_n(y) -f_n(x))(\varphi(y) -\varphi(x))dydx
\end{align*}
which proves \eqref{conv.h1}. It remains to prove the claim \eqref{claim.1}. To this end notice that
\begin{align*}
\label{}
  \int_D\int _{\rr^d}\rho(z)& \Big| \frac{\varphi(x+\eps z)-\varphi(x)}{\eps}\chi(x+\eps z)-z\nabla \varphi(x) \Big|^2dz dx\\
  \lesssim  & \int_D\int _{\rr^d}\rho(z)  \Big| \frac{\varphi(x+\eps z)-\varphi(x)}{\eps}-z\nabla \varphi(x) \Big|^2\chi(x+\eps z)dz dx \\
  &+\int_D\int _{\rr^d}\rho(z) |z\nabla \varphi(x)|^2  |\chi(x+\eps z)-1|^2dz dx.
\end{align*}
The last term goes to zero thanks to the fact that $|z|^2\rho(z)|\nabla \varphi(x)|^2\in L^1(D\times \rr^d)$ and the dominated convergence theorem. For the first term we use the Fourier transform and again the dominated convergence theorem to obtain that
\begin{align*}
\label{}
   \int_D\int _{\rr^d}\rho(z) & \Big| \frac{\varphi(x+\eps z)-\varphi(x)}{\eps}-z\nabla \varphi(x) \Big|^2 dz dx\\
   &\leq \int _{\rr^d}\int _{\rr^d} \rho(z) \int_{\rr^d} \Big| \frac{e^{2\pi i \eps \xi z}-1}{\eps }-2\pi i \xi z\Big|^2 |\hat \varphi (\xi)|^2d\xi\rightarrow 0, \qquad \eps \rightarrow 0.
\end{align*}
The proof is now finished.	
\end{proof}

In our analysis we need a version of the last lemma that will involve also integrals in time and implies convergences in $L^2((0,T)\times D)$. This is inspired in \cite{IgnatIgnat}. The proof follows the same ideas of the previous lemma and hence it is omitted.

\begin{lemma}
	\label{convergence.h1.time} Let $D\subset \rr^d$ be an open set with the extension property. Let $\rho:\rr^d\rightarrow\rr$ be a nonnegative $L^1$ radial function having a second momentum in $L^1(\rr^d)$, $\rho>0$ in neighbourhood of $x=0$, and $\rho_n(x)=n^d\rho(nx)$. Let $f_n$ be a sequence in $L^2((0,T)\times D)$ such that
	\begin{equation}
\label{est.1.878}
  n^2\int _0^T \int_D \int_D \rho_n(x-y)(f_n(x)-f_n(y))^2dxdy\leq M.
\end{equation}
 1. If $f_n\rightharpoonup f $ in $L^2((0,T)\times D)$ then $f\in L^2((0,T),H^1(D))$
  and
  \[
  F_n(x,z)=(\rho(z))^{\frac{1}2}\chi_D(x+\frac zn) \frac {f_n(x+\frac zn) -f_n(x)}{\frac 1n}\rightharpoonup (\rho(z))^{\frac{1}2} z\cdot \nabla f(x)
  \]
  weakly in $L^2((0,T),L_x^2(D)\times L^2_z(\rr^d))$.\\
2. For any $f\in L^2((0,T),H^1(D))$
\begin{equation}
\label{conv.liminf.00}
  \int_{\rr^d}\rho(z)|z|^2dz \int _0^T \int _{D}\nabla f(x)\nabla \varphi(x)dx\leq \liminf_{n\to \infty}  n^2\int _0^T \int_D \int_D \rho_n(x-y)(f(x)-f(y))^2dxdy
\end{equation}
3. If $\varphi\in L^2((0,T),H^1(\rr^d))  $ then
  \begin{align}
\label{conv.h1.99}
n^2   \int_0^T\int_D \int_D \rho_n(x-y)(f_n(x)-f_n(y))&(\varphi(x)-\varphi(y))dxdy\\
&\rightarrow
\nonumber \int_{\rr^d}\rho(z)|z|^2dz \int_0^T \int _{D}\nabla f(x)\nabla \varphi(x)dx.
\end{align}
{Notice that by 1. we have $f\in L^2((0,T),H^1(D))$.}

4. If $D$ is a smooth bounded domain of $\rr^d$, $\rho(x)\geq \rho(y)$ if $|x|\leq |y|$ and
$
\|\partial _t f_n\|_{L^2((0,T),H^{-1}(D))}$
is uniformly bounded
then $\{f_n\}_{n}$ is relatively compact in $L^2(D)$.
\end{lemma}

Finally, we include a lemma in the one dimensional case.

\begin{lemma}
\label{adjoint.intervals}
Let $\rho$ a nonincreasing $L^1$ radially symmetric function with $|r|^2\rho(r)\in L^1(\rr)$. Let $-\infty\leq a<0<b\leq +\infty$.
Then, for any $\varphi\in H^1(a,b)$, it holds that
\begin{equation}
\label{limit.disjoint.intervals}
 \eps^{-2} \lim_{\eps \rightarrow 0} \int _a^0 \int _0^b \rho_\eps(x-y)(\varphi(x)-\varphi(y))^2dxdy=0.
\end{equation}
\end{lemma}
\begin{proof}We use \eqref{conv.liminf.00} (see also \cite[Th. 2, Remark 5]{MR1942116})
to obtain that
	\begin{equation}
\label{lema.brezis}
 \| r^2\rho(r) \|_{L^1(\rr)}\int_{I} f_x^2 \leq \liminf_{\eps \rightarrow 0} \eps^{-2}\int_{I}\int_{I} \rho_\eps(x-y)(f(x)-f(y))^2dxdy,
\end{equation}
	holds for for any  interval $I$ of $\rr$ and $f\in H^1(I)$. Moreover, applying  \eqref{conv.h1.99}  to $f_n=\varphi=f$ we have equality when $I=\rr^d$.
	
	Since we are in dimension one,  any  function $\varphi\in H^1(a,b)$ can be extended to a function $\widetilde \varphi\in H^1(\rr)$. Then
	\begin{align*}
\label{}
   2\eps^{-2}  \int _a^0 \int _0^b &\rho_\eps(x-y)(\varphi(x)-\varphi(y))^2dxdy\leq
     2\eps^{-2}  \int _{-\infty}^0\int _0^{\infty} \rho_\eps(x-y)(\widetilde\varphi(x)-\widetilde\varphi(y))^2dxdy\\
     =&\eps^{-2}\int_\rr \int_\rr  \rho_\eps(x-y)(\widetilde\varphi(x)-\widetilde\varphi(y))^2dxdy\\
     & - \eps^{-2}\int _{-\infty}^0 \int _{-\infty}^0  \rho_\eps(x-y)(\widetilde\varphi(x)-\widetilde\varphi(y))^2dxdy \\
     &-\eps^{-2} \int _0^\infty \int _0^\infty \rho_\eps(x-y)(\widetilde\varphi(x)-\widetilde\varphi(y))^2dxdy.
\end{align*}
Using \eqref{lema.brezis} with $I=(-\infty,0)$ and $I=(-\infty,0)$ as well as the fact that it becomes equality when $I=\rr$ we get that
\begin{align*}
  \limsup_{\eps\rightarrow 0}  2\eps^{-2} & \int _a^0 \int _0^b \rho_\eps(x-y)(\varphi(x)-\varphi(y))^2dxdy\\
  &\leq
 \| r^2\rho(r) \|_{L^1(\rr)}\Big[\int_\rr \varphi_x^2-\int_{-\infty}^0 \varphi_x^2 -\int_0^\infty \varphi_x^2
\Big]=0,
\end{align*}
which finishes the proof.
\end{proof}

{\bf Acknowledgements.}

Part of this work was done during a visit of JDR to the Institute of Mathematics ``Simion Stoilow'' at Bucharest,
he is grateful for the friendly and stimulating working atmosphere found there.

 L. I. was partially supported by a grant of Ministry of Research and Innovation, CNCSUEFISCDI, project PN-III-P1-1.1-TE-2016- 2233, within PNCDI III. J.D.R. partially supported by CONICET grant PIP GI No 11220150100036CO (Argentina), PICT-2018-03183 (Argentina) and UBACyT grant 20020160100155BA (Argentina).

\end{document}